\documentclass{article}
\usepackage{graphicx}
\usepackage{Zoe_C_Style}
\usepackage{tikz}

\title{Equivariant Cosheaves and Finite Group Representations in Graphic Statics}

\author{Zoe Cooperband\thanks{Department of Electrical and Systems Engineering (ESE), GRASP Lab, and Polyhedral Structures Lab, University of Pennsylvania},$\,$ Miguel Lopez\thanks{Department of Applied Mathematics and Computational Science (AMCS), University of Pennsylvania} $\,$and  Bernd Schulze\thanks{Department of Mathematics and Statistics, Lancaster University}}

\date{}

\begin{document}

\maketitle

\begin{abstract}
This work extends the theory of reciprocal
diagrams in graphic statics to frameworks that are invariant under finite group actions by utilizing the homology and representation theory of cellular cosheaves, recent tools from applied algebraic topology. By introducing the structure of an equivariant cellular cosheaf, we prove that pairs of self-stresses and reciprocal diagrams of symmetric frameworks are classified by the irreducible representations of the underlying group. We further derive the symmetry-aligned Euler characteristics of a finite dimensional equivariant chain complex, which for the force cosheaf yields a new formulation of the symmetry-adapted Maxwell counting rule for detecting symmetric self-stresses and kinematic degrees of freedom in frameworks. A freely available program is used to implement the relevant cosheaf homologies and illustrate the theory with examples.
\end{abstract}

%%%%%%%%%%%%%%%%%%%%%%%%%%%%%%%%%%%%%%%%%%%%%%%%%%
\section{Introduction}
%%%%%%%%%%%%%%%%%%%%%%%%%%%%%%%%%%%%%%%%%%%%%%%%%%

{\it Graphic statics} is a geometric toolbox for analysing the relationship between the form of a bar-joint framework and its internal force loading. The theory dates back to classical work of Maxwell and Cremona from the 19th century \cite{maxwell1864xlv,maxwell1870reciprocal,cremona} and has been widely used for designing and modelling various types of real-world structures. Building on earlier work by Rankine \cite{ran} and others, they discovered that from a self-stress (or equilibrium stress) in a planar framework one can construct a reciprocal diagram (or "force diagram", as it is usually called by engineers), which realises the dual of the original graph as a framework with bar lengths  determined by the stress coefficients. See Figure~\ref{fig:triangular_recip} for an example. They also found that a self-stress over a planar framework induces a vertical polyhedral lifting of the framework into $3$-space known as a ``discrete Airy stress function polyhedron'', with stress coefficients encoded by changes in slope.

In 1982 Whiteley established the converse of these results by showing that every reciprocal diagram and every polyhedral lifting corresponds to a self-stressed framework in the plane \cite{WW82}. See also Crapo and Whiteley \cite{CW93}. Since then, the Maxwell-Cremona correspondence has been applied to the solution of numerous problems in  geometric rigidity theory, polyhedral combinatorics and discrete and computational geometry. See e.g. \cite{Con03,EadesGarvan,hopcroft1992paradigm,streinu,streinuerr}. The theory is also closely linked  to the duality between Voronoi diagrams and Delaunay tessellations (see e.g. \cite{WAC}).

Increased interest in the design of material-efficient structures in engineering, such as gridshell roofs or cable net structures, has heralded a surge of interest in this theory. The visual nature of force-form duality allows an integrated analysis early into the design process, crucial for finding optimal designs. Modern computational tools play a key role here, as they allow quick visualisations of these relationships. Techniques from graphic statics have recently also found new applications in control theory \cite{nguyen2017,nguyen2017convex} and materials science \cite{hab}.

%nature The force diagrams of the frameworks (or "form diagrams") obtained by vertically projecting the skeleton of the roof structure into the plane a

Various generalisations of the Maxwell-Cremona correspondence have recently been established. For example, while a 3D version of the reciprocal diagram was already described by Rankine and Maxwell, the availability of modern computational tools has recently motivated the engineering community to investigate this further \cite{marina,dacunto}. Other recent results include extensions to non-planar frameworks \cite{ErLin,karp}, periodic frameworks \cite{BorStr} and higher-dimensional polytopal complexes \cite{Rybnikov99,Rybnikov991,karp}.

Recently, there has been an explosion of results regarding the rigidity, flexibility and stressability of \emph{symmetric} frameworks. See e.g. \cite{connelly_guest_2022,KTHandbook,SWbook} for an introduction to the theory and a summary of results. Much of this work has been motivated by problems from the applied sciences and industry, where symmetry is utilized by both man-made and natural structures for stability, construction, and aesthetics. For gridshell roofs, %each state of self-stress of the form diagram relates to a funicular gravity loading of the gridshell, i.e. the applied loads are taken through axial forces only, so that there is no bending moment in the gridshell. Thus, 
increasing the dimensionality of the space of self-stresses can reduce the volume of material needed to construct the load-bearing  structure \cite{schulzeetal}. %, because self-stresses in the form diagram relate to the ability of the structure to resolve applied loads  through axial forces only. Moreover, %both "fully-symmetric" self-stresses (i.e., self-stresses with the same stress coefficient on all bars in the same orbit under the group action)  and anti-symmetric self-stresses (i.e. self-stresses for which the stress coefficients satisfy certain sign patterns with respect to the symmetry group) 
By associating the self-stresses of different symmetry types to irreducible representations of the symmetry group of the structure\footnote{Self-stresses with both symmetric and anti-symmetric sign patterns for the stress coefficients play an important role in the design of gridshells. While fully-symmetric self-stresses  relate to  symmetric vertical loadings of the gridshell  (such as self-weight), anti-symmetric  self-stresses  relate to anti-symmetric loadings, such as uneven gravity loads arising from  drifted snow, for example.}, new tools for designing and analysing gridshell structures were recently obtained in \cite{schulzeetal,msmb}. Using a similar approach, refined relations between self-stresses and motions of different symmetries in form and force diagrams were established in \cite{SCHULZE2023112492}.

\begin{figure}
    \centering\includegraphics{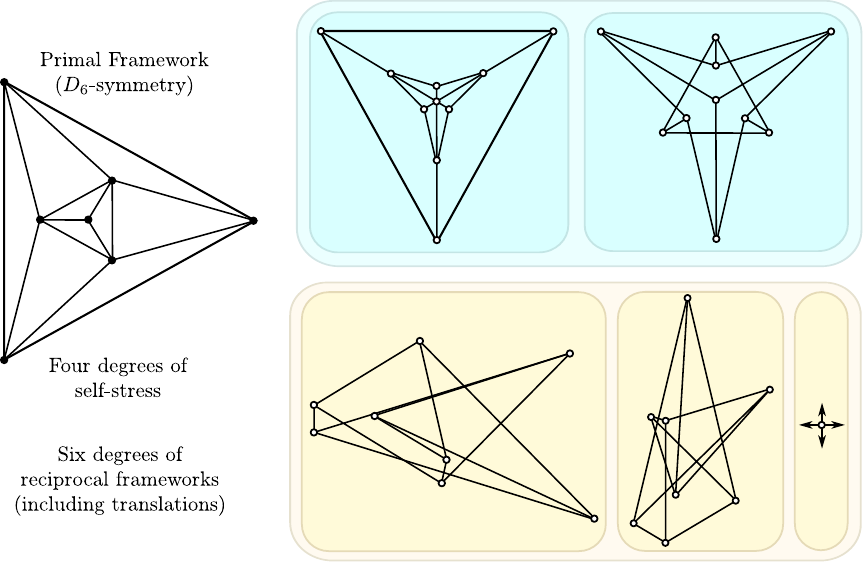}
    \caption{A planar framework with $D_6$-symmetry is pictured (left). There are six dimensions of reciprocal diagrams (right), including two trivial translation dimensions. Each reciprocal diagram is an assignment of geometric coordinates to the dual graph, constrained such that dual edges must be parallel to their primal counterparts. Dual nodes may be drawn overlapping, as is exemplified in the trivial reciprocal diagram (bottom right). The space of reciprocal diagrams is subdivided by symmetry (corresponding to irreducible representations of $D_6$).}
    \label{fig:triangular_recip}
\end{figure}

Recently, a homological description of 2D graphic statics was given in \cite{cooperband2023homological} affirming algebraic topology as a useful theory for structural engineering. This work built upon observations in \cite{cooperband2023homological} that the self-stresses of a framework can be encoded with a tool called a cellular cosheaf. Developed and popularized by the theses of Shepard \cite{shepard1986cellular} and Curry \cite{curry2014sheaves}, cellular sheaves and cosheaves have rapidly found wide application in network coding \cite{curry2014sheaves, ghrist2011network}, optimization \cite{hansen2019distributed, moy2020path}, space networks \cite{short2022current, short2021towards}, opinion dynamics \cite{hansen2021opinion}, and more. Recently, there have been numerous exciting applications in machine learning \cite{hansen2020sheaf, bodnar2022neural, barbero2022sheaf, battiloro2023tangent} where the sheaf Laplacian \cite{hansenSpectralTheoryCellular2019a, hansen2020laplacians} fittingly encodes network data diffusion. This work incidentally describes data equivariance over networks and cell-complexes, bridging the gap towards advances in {\it equivariant} and {\it convolutional} neural networks \cite{cohen2016group, bronstein2021geometric}.

Cosheaves embody the {\it finite element} approach, namely breaking a physical system into smaller units which relate to other units by constraint equations. Finite elements, when framed in terms of cosheaf stalks and extension maps, have access to the wide array of formal methods available in homological algebra \cite{Weibel1994}. The concept of moving from local to global equilibria is paralleled with moving from local to global sections in sheaves, bundles, and other algebraic constructions. This paper is an early description of such methods for {\it homological engineering}, or designing chain complexes and homology spaces to model the constraints and degrees of freedom of complex physical systems.

Cellular cosheaves are a valuable tool in the rapidly expanding field of topological data analysis \cite{curry2014sheaves, yoon2018cellular}. Persistent homology, the core method to analyze and distill topological information from data, has been applied to pursuit and evasion games \cite{de2007coverage, adams2015evasion}, robot path planning \cite{bhattacharya2015persistent}, material science \cite{obayashi2022persistent, onodera2019understanding}, neuroscience \cite{sizemore2019importance, kanari2018topological, bardin2019topological}, and protein folding \cite{xia2014persistent, kovacev2016using}. This work extends the reach of computational homology towards group symmetries and structural engineering.

\subsection{Contributions}
In the present article, we utilize representation theory and computational homology to gain deeper insights into 2D symmetric graphic statics. We show that group actions on frameworks give rise to group actions on cosheaves whose homology encodes structural information. The {\it force and position cosheaves}, encoding the space of self-stresses and reciprocal coordinates respectively, both pass to the group action, demonstrated here with cyclic and dihedral groups. We develop the {\it Euler characteristic} of irreducible cosheaf characters to reformulate Maxwell's counting rule for symmetric frameworks. We then prove that the Maxwell-Cremona equivalence between self-stresses of a framework and its  reciprocal diagrams not only occurs over symmetric frameworks, but that this equivalence also respects the underlying irreducible representations.

\begin{theorem}[\ref{thm:G_GS}) (Symmetric Planar Graphic Statics]
    For every symmetric planar $G$-framework in $\R^2$ and for every irreducible  representation $\mu^{(j)}$ of $G$, there is an isomorphism between $\mu^{(j)}$-symmetric self-stresses and $\mu^{(j)}$-symmetric reciprocal diagrams up to $\mu^{(j)}$-translation symmetry.
\end{theorem}

As a consequence of Theorem~\ref{thm:G_GS}, self-stress and reciprocal diagram pairs can be aligned and organized by their underlying symmetry. Moreover, this theory can be used to decompose reciprocal diagrams of $G$-symmetric self-stressed  frameworks into diagrams of basic symmetry types corresponding to irreducible representations of $G$. See  Figure~\ref{fig:klien_recip} for an example.

Equivariant cellular cosheaves, as introduced here, are closely related to equivariant coefficient systems over simplicial complexes as developed in \cite{boltje1994identities}. Equivariant sheaves have found use in equivariant homotopy theory \cite{EquivShvAndFunctors}, and equivariant chain complexes and homology is an active sub-field of algebraic topology \cite{MayEquivariant}, but to the authors knowledge, this is the first description and application of this particular structure.

\begin{figure}[htp]
    \centering
     \begin{subfigure}[b]{0.3\textwidth}
         \centering
         \includegraphics[width=\textwidth]{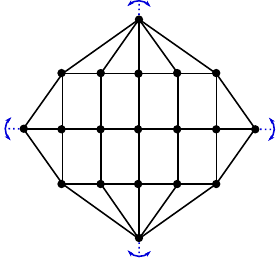}
         \caption{A primal planar graph with $D_4\iso \Z_2 \times \Z_2$ symmetry.}
     \end{subfigure}
     \hfill
     \begin{subfigure}[b]{0.65\textwidth}
         \centering
         \includegraphics[width=\textwidth]{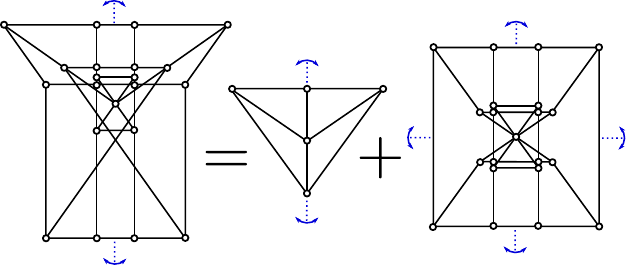}
         \caption{Reciprocal diagrams decompose by symmetry. Dual nodes may coincide and edges may cross in the dual graph.}
     \end{subfigure}
    \caption{A planar framework (geometric graph) with $D_4$ symmetry is pictured (a). This framework may be considered as a form diagram of a small gridshell roof, as it has no edge crossings and exhibits some other key desirable features, such as quad-dominance, aligned vertices and dihedral symmetry. Every reciprocal framework can be subdivided into a linear combination of component frameworks exhibiting symmetries of different irreducible representations of $D_4$, in the sense that adding the point coordinates of the diagrams yields the original reciprocal diagram (b). 
    }
    \label{fig:klien_recip}
\end{figure}

The paper is organised as follows: In Sections~\ref{sec:statics} and \ref{sec:reps} we  provide the necessary background on graphic statics, cellular cosheaves and group representation theory.  In Section \ref{equivariantcosheaf} we introduce the notion of an \emph{equivariant} cellular cosheaf, combining the theories of the previous two sections. Section~\ref{sec:irrid_homology} focuses on the irreducible components of the homology relations, and in particular develops the graphic statics relations of irreducibles. Finally, in Section~\ref{sec:future} we discuss some avenues for future work.

The figures in this paper were developed using Python code that is freely available at the following link:
\hyperlink{https://github.com/zcooperband/EquivariantGraphicStatics}{https://github.com/zcooperband/EquivariantGraphicStatics}. In the linked project, the relevant cosheaf homologies are implemented using matrix methods, and equivariant irreducible self-stress reciprocal diagram pairs are extracted. The project contains several prepared sample frameworks and tests for quick verification of commutativity of diagrams of the form~\eqref{eq:chain_cube}.

%%%%%%%%%%%%%%%%%%%%%%%%%%%%%%%%%%%%%%%%%%%%%%%%%%
\section{Statics, Graphic Statics, and Cosheaves} \label{sec:statics}
%%%%%%%%%%%%%%%%%%%%%%%%%%%%%%%%%%%%%%%%%%%%%%%%%%

While in structural engineering it is common to merge the abstract/combinatorial and geometric characteristics of a pin-jointed truss model, it is useful to distinguish between the two. Discounting geometric singularities, many algebraic properties of a truss are invariant under changes in geometry. The underlying combinatorial structure is that of a cell complex.

A (finite) \emph{cell complex} is a topological space $X$ partitioned into a finite number of cells $\{c\}$, where each cell $c$ is homeomorphic to a topological disk of some dimension\footnote{
Moreover, cells must have ``nice intersections"; for a complete definition of a regular (CW) complex see \cite{curry2014sheaves}.  We note that systems of polyhedra form regular cell complexes.
}. We say cells $c$ and $d$ are  {\it incident} and write $c\lhd d$ if $c$ is a lower dimensional cell on the boundary of the closure of $d$.

A {\it signed incidence relation} on a cell complex $X$ is a pairing $[-:-]: X \times X \to \{0, \pm1\}$ satisfying the following properties:
\begin{itemize}
    \item (Adjacency) $[c:d]\neq 0$ if and only if $c\lhd d$ and $\dim c + 1 = \dim d$.
    \item (Directed Edges) $[u:e][v:e] = -1$ for an edge with incident vertices $u,v\lhd e$.
    \item (Regularity) For any $b\lhd d$, $\sum_{c} [b : c][c : d] = 0$.
\end{itemize}
A signed incidence relation encodes the consistency
%agreement or disagreement\miguel{consistency?} 
of local orientations over cells. When the orientation of two incident cells $c\lhd d$  agree, we set $[c:d]=+1$; otherwise we set $[c:d] = -1$. These properties are essential and will be used widely later.

A {\it framework} $(X,p)$ in the plane $\R^2$ a cell complex $X$ together with a {\it realization} map $p:V(X) \to \R^2$ that assigns each vertex $v$ a geometric position $p_v = p(v)\in \R^2$. We typically require that $p$ be locally injective, meaning that for every edge with endpoints $u,v \lhd e$, we have $p_u\neq p_v$. A framework $(X,p)$ is {\it planar} if its edges, embedded in $\R^2$ as straight lines, do not intersect anywhere except at their endpoints. For a planar framework, the faces of the cell complex $X$ naturally correspond to the connected components of the complement of the union of these geometric edges.

Every (not-necessarily planar) framework models a {\it pin-jointed truss} in the plane. This is a geometric model of a physical truss, where the nodes allow rotations in any direction of the space and the truss members are loaded in pure axial tension or compression. A {\it stress} over a framework $(X,p)$ is an assignment of a real-valued scalar $w_e$ to each edge $e$ encoding the internal tension or compression force over that edge. A {\it self-stress} (or {\it equilibrium stress}) $w$ on a framework $(X,p)$ is a stress that satisfies the following equation\footnote{
Here the value $w_e$ corresponds to a compression force scaled by the length of the edge $\|p_v - p_u\|$.
} at every vertex $v$
\begin{equation}\label{eq:equilibrium_stress}
    \sum\limits_{\{v\lhd e \rhd u : v\neq u\}} w_e (p_v - p_u) = 0,
\end{equation}
the sum being over the vertices and edges in the one-hop neighborhood of $v$.

Equilibrium stresses encapsulate the condition that the truss is in force {\it equilibrium} and are of vital importance to structural engineering statics. In the next section, we show that equilibrium stresses are an instance of a much more general phenomenon.

%%%%%%%%%%%%%%%%%%%%%%%%%%%%%%%%%%%%%%%%%%%%%%%%%%
\subsection{Cellular Cosheaves}
%%%%%%%%%%%%%%%%%%%%%%%%%%%%%%%%%%%%%%%%%%%%%%%%%%

It is typical in a variety of engineering applications to assign vector valued data to the cells of a cell complex. In mechanical systems, this data can consist of forces, kinematic motions, positions, and other geometric-algebraic data. Cosheaves are mathematically precise formulations of these distributed data structures. 

\begin{definition}[Cellular Cosheaf]\label{def:cosheaf}
    Fixing a field $k$,  a ($k$-valued cellular) {\it cosheaf} $\calk$ consists of the following data. Each cell $c \in X$ is assigned a finite dimensional vector space $\calk_c \iso k^n$  called the {\it stalk at} $c$. When cells $c \lhd d$ are incident, a linear {\it extension map} is assigned between stalks $\calk_{d \rhd c}: \calk_d \to \calk_c$. Cosheaves are functors, meaning that $\calk_{c\rhd b} \circ \calk_{d\rhd c} = \calk_{d\rhd b}$ and $\calk_{c\rhd c} = \id$ for incident cells $b\lhd c\lhd d$.
\end{definition}
We will fix the field of all cellular cosheaves to be $\R$ or $\C$.

One thinks of a cosheaf as a blueprint for local algebraic data, describing what data is attached to which cell and how these relate. However, to detect global algebraic structure, we use the blueprint and build a computational machine known as a chain complex.

\begin{definition}[Chain complex]
    The space of $i$-{\it chains} of a cosheaf $\calk$ over a cell complex $X$ is the direct sum of stalks
    \[ C_i\calk = \bigoplus_{\dim c = i} \calk_c.\]
    The {\it boundary} of an $i$-chain $x$ is the $i-1$ dimensional chain $\bdd x$ with component
    \[(\bdd x)_c = \sum_{c\lhd d} [c:d] \calk_{d\rhd c} x_d\]
    at the $i-1$ dimensional cell $c$. Note that we define $\bdd x = 0$ if $x$ is a 0-chain. From the regularity property of signed incidence relations, it follows that $\bdd \circ \bdd = 0$ and hence
\begin{equation}\label{eq:chain_complex}
    \dots C_{i+1} \calk \xrightarrow{\bdd_{i+1}} C_i \calk \xrightarrow{\bdd_i} C_{i-1} \calk \to \dots
\end{equation}
is a {\it chain complex} denoted $C\calk$.
\end{definition}

Certain practical necessities, such as the orientations of cells, are necessary for computations but do not matter in the grand scheme of things. This is motivation for decoupling the core abstraction of the cosheaf from its computational aspects of its chain complex. The core use of cosheaf chain complexes is to compute {\it homology}.

We say an $i$-chain $x\in C_i\calk$ is a {\it cycle} if $\bdd_i(x) = 0$. The $i$-th homology of the chain complex~\eqref{eq:chain_complex} is given by
\begin{equation*}
    H_i \calk = \ker \bdd_i / \im \bdd_{i+1}
\end{equation*}
i.e. the space of quotients of cycles by boundaries of higher dimensional chains. The homology of a cosheaf is invariant of choices of (valid) signed incidence relations.

\begin{example}[Constant cosheaves]
    \label{ex:constant}
    Let $V$ be a finite dimensional vector space. The {\it constant cosheaf} $\overline{V}$ over a cell complex $X$ has identical stalks $\overline{V}_c = V$ for all cells $c$ and identity extension maps $\overline{V}_{d\rhd c} = \id$.   
    The homology of $\overline{V}$ is equivalent to the cellular homology of $X$. In particular $H_i\overline{k} = H_i(X;k)$ in $k$-valued coefficients. More generally when $V$ is of dimension $m$ we have $\dim H\overline{V}=m\dim H(X;k)$.
\end{example}

\begin{example}[The force cosheaf]
    \label{ex:force}
    The following force cosheaf encodes the forces within an axially loaded truss as well as equilibrium stresses of the truss. This cosheaf has been defined and developed previously in the context of graphic statics \cite{cooperband2023homological, cooperband2023reciprocal}.
    
    Fix a framework $(X,p)$ in $\R^2$. The force cosheaf $\calf$ over $(X,p)$ has stalks $\calf_e=\R$ encoding the axial force in an edge $e$ and $\calf_v=\R^2$ encoding the space of external forces at each joint. Both extension maps $\calf_{e\rhd u}$ and $\calf_{e\rhd v}$ send $1\in \calf_e$ to the same vector\footnote{
    This assignment dictates that $1\in \calf_e$ corresponds to the edge being in compression. If we were to wish for the basis element $1\in \calf_e$ to correspond to a tension value, we simply set $\calf_{e\lhd v}(1) = \calf_{e\lhd u}(1) = [v:e](p_u-p_v)$.}
    \begin{equation}\label{eq:force_extension}
        [v:e](p_v - p_u) = [u:e](p_u - p_v).
    \end{equation}
    
    The boundary map of the force cosheaf $\bdd:C_1\calf \to C_0 \calf$ can be represented as a size $n|V|\times |E|$ matrix known as the {\it equilibrium matrix}. Note that the equilibrium matrix is the transpose of the classical rigidity matrix from geometric rigidity theory \cite{connelly_guest_2022,WW}. The kernel of $\bdd$ is the first homology $H_1 \calf$, the vector space of equilibrium stresses of the structure. To confirm this, we expand the boundary matrix at a chain $w\in C_1 \calf$
    \begin{equation*}
        (\bdd w)_v
        = \sum\limits_{v\lhd e} [v:e] \calf_{e\rhd v} w_e
        = \sum\limits_{\{ v\lhd e\rhd u : v\neq u \}} [v:e]^2 (p_v - p_u) w_e
    \end{equation*}
    which is zero at all vertices $v$ exactly when $w$ is an equilibrium stress following equation~\eqref{eq:equilibrium_stress}. The cokernel of $\bdd$ is the zeroth homology $H_0 \calf = C_0 \calf / \im \bdd$, interpreted as the space of {\it infinitesimal motions} of the framework. These include the trivial infinitesimal rigid body motions (rotation and translation) as well as the non-trivial infinitesimal motions (mechanisms).
\end{example}
%%%%%%%%%%%%%%%%%%%%%%%%%%%%%%%%%%%%%%%%%%%%%%%%%%
\subsection{Maps Between Cosheaves}
%%%%%%%%%%%%%%%%%%%%%%%%%%%%%%%%%%%%%%%%%%%%%%%%%%

An {\it exact sequence} of vector spaces is a sequence of vector spaces and linear maps
\begin{equation}\label{eq:ses}
    \dots \to V_3 \xrightarrow{f_2} V_2 \xrightarrow{f_1} V_1 \xrightarrow{f_0} V_0 \xrightarrow{f_{-1}} V_{-1} \to \dots
\end{equation}
where for each index the maps satisfy $\im f_i = \ker f_{i-1}$. A {\it short exact sequence} is of form~\eqref{eq:ses} where $V_i = 0$ except at indices $i=0,1,2$. In this case, the map $f_1$ is injective and the map $f_0$ is surjective. Furthermore there are isomorphisms $V_0 \iso V_1/V_2$ and $V_0 \oplus V_2 \iso V_1$

Let $\calk$ and $\call$ be cosheaves over a cell complex $X$. A cosheaf map $\phi: \calk \to \call$ is comprised of component maps between stalks $\phi_c: \calk_c \to \call_c$ such that the following diagram commutes
\begin{equation}\label{eq:Gcosheaf_map}
    \begin{tikzcd}
        \calk_d \ar[r, "\phi_d"] \ar[d, "\calk_{d\rhd c}"] & \call_d \ar[d, "\call_{d\rhd c}"] \\
        \calk_c \ar[r, "\phi_c"] & \call_c
    \end{tikzcd}
\end{equation}
for every pair of incident cells $c\lhd d$. Cosheaf maps induce maps on chain complexes $\phi: C\calk \to C\call$ comprised of the constituent maps.

\begin{definition}
    A short exact sequence of cosheaves
    \begin{equation*}
        0 \to \calk \to \call \to \calm \to 0
    \end{equation*}
    has an induced short exact sequence of cosheaf maps 
    \begin{equation}\label{eq:chain_ses}
        0 \to C\calk \xrightarrow{\phi} C\call \xrightarrow{\psi} C\calm \to 0
    \end{equation}
    such that the induced sequence at each stalk
    \begin{equation*}
        0 \to \calk_c \xrightarrow{\phi_c} \call_c \xrightarrow{\psi_c} \calm_c \to 0
    \end{equation*}
    is exact. All cosheaf stalks we will discuss are finite dimensional and we will assume that the underlying cell complex $X$ only has a finite number of cells, so the sequence~\eqref{eq:chain_ses} is an exact sequence of finite dimensional vector spaces.
\end{definition}

In particular, at each stalk $\calm_c \iso \call_c / \calk_c$ and therefore $C \calk \oplus C\calm \iso C\call$ is an isomorphism of chain complexes.

From any injective cosheaf map $\phi: \calk \to \call$ we can construct the {quotient cosheaf} $\call/\phi\calk$ with stalks $(\call/\phi\calk)_c = \call_c / \im \phi_c \calk_c$. If each stalk $\calk_c$ is considered as a subspace of the stalk $\call_c$ under the embedding $\phi_c$, we may treat $\calk$ as a sub-cosheaf of $\call$ and drop notation declaring the quotient cosheaf to be $\call/\calk$.

A short exact sequence of cosheaf maps induces maps on homology, and moreover we get a  {\it long exact sequence in homology}
\begin{equation}\label{eq:les}
    \dots \to H_{i+1} \call / \calk \xrightarrow{\vartheta} H_i \calk \xrightarrow{\phi} H_i\call \xrightarrow{\psi} H_i \call / \calk \xrightarrow{\vartheta} H_{i-1} \calk \to \dots
\end{equation}
where $\vartheta$ are {\it connecting homomorphisms}.

\begin{example}[Planar graphic statics]
\label{ex:graphic_statics}
    Suppose $(X,p)$ is a planar framework in $\R^2$. Both the force cosheaf $\calf$ introduced in Example~\eqref{ex:force} and the constant cosheaf $\overline{\R^2}$ may be defined over this same framework $(X,p)$. There is a natural injective cosheaf map $\phi$ from $\calf$ to $\overline{\R^2}$ which we now describe.

    The map $\phi$ is the identity linear map over vertices, with $\phi_v: \calf_v \to \overline{\R^2}_v$, a canonical isomorphism equating stalks $\calf_v = \R^2$ and $\overline{\R^2}_v = \R^2$. The map $\phi$ is injective over edges; for an edge $u,v\lhd e$,  $\phi_e: \calf_e \to \overline{\R^2}$ sends $1$ to the normalization of the vector $[u:e](p_u - p_v)$ (this is the same as the extension maps $\calf_{e\rhd u}$ and $\calf_{e\rhd v}$). By construction, $\phi$ is natural, with
    \begin{equation*}
        \id \circ \calf_{e\rhd v} = \phi_v \circ \calf_{e\rhd v} = \overline{\R^2}_{e\rhd v} \circ \phi_e = \id \circ \phi_e
    \end{equation*}
    with a similar equation over the other incidence $u\lhd e$.

    We assign the notation $\calp :=\overline{\R^2} / \phi \calf$ for convenience, and let $\pi:\overline{\R^2} \to \calp$ denote the cosheaf quotient map. We call $\calp$ the {\it position cosheaf} dual to $\calf$ because we will see that $\calp$ encodes the positions of the dual vertices of reciprocal diagrams. Since trusses have trivial data over faces $f$, $\calf_f = 0$ and $\calp_f = \overline{\R^2}_f = \R^2$. Over an edge $u,v\lhd e$, we know that $\phi\calf_e$ is the span of the vector $(p_u - p_v)$ and consequently $\calp_e = \R^2 / \text{span}\{p_u-p_v\}$. Lastly, $\calp_v=0$ over vertices $v$.

    In graphic statics, self-stresses of a truss $(X,p)$ are associated with {\it reciprocal diagrams}, here realizations $\xi$ of the dual graph $\tilde{X}$ such that an edge $e$ is parallel with its dual $\tilde{e}$. See Figures~\ref{fig:triangular_recip} and~\ref{fig:mirror_gs}. Abstractly, the position of a dual node $\tilde{f}$ is a coordinate in $\R^2$ which we encode by the map $\xi:\tilde{F} \to \R^2$. The collection of these positions $\xi\tilde{F}$ must satisfy the constraint that $\xi_{\tilde{f}} - \xi_{\tilde{g}}$ is parallel with the vector $p_u - p_v$ for any edge with vertex and face incidences $u,v\lhd e\lhd f,g$. Equivalently, $\xi_{\tilde{f}} - \xi_{\tilde{g}}$ is an element of $\text{span}\{p_u-p_v\}$ and therefore
    \begin{equation}\label{eq:reciprocal_constraint}
        [\xi_{\tilde{f}} - \xi_{\tilde{g}}] = [0] \in \R^2 / \text{span}\{p_u - p_v\} = \calp_e
    \end{equation}
    is the zero class.

    This is all to say the space of parallel reciprocal diagrams, encoded by realizations $\xi$ over $\tilde{F}\iso F$, are elements of $H_2 \calp$. We see that $\xi$ is a cycle if and only if Equation \eqref{eq:reciprocal_constraint} is satisfied everywhere \cite{cooperband2023homological}.

    From the exact sequence of $\calf$, $\overline{\R^2}$ and the quotient cosheaf $\calp$ we have a segment of the long exact sequence
    \begin{equation}\label{eq:reciprocal_les}
        0 \to H_2 \overline{\R^2} \to H_2 \calp \xrightarrow{\vartheta} H_1 \calf \xrightarrow{\phi} H_1 \overline{\R^2} \to \dots 
    \end{equation}
    We know that since $X$ has spherical topology, $H_2 X \iso \R$ and $H_1 X = 0$. Consequently, the constant homology is determined as $H_2 \overline{\R^2} \iso \R^2$ and $H_1 \overline{\R^2} = 0$ and line~\eqref{eq:reciprocal_les} simplifies considerably. We find there is an isomorphism $\vartheta: H_2 \calp/\R^2 \to H_1 \calf$, meaning that the space of self-stresses of $X$ is isomorphic to the space of parallel reciprocal diagrams of $\tilde{X}$ up to global translation. 
\end{example}

In the above Example~\ref{ex:graphic_statics}, we derived the characteristics of the position cosheaf purely from the force and constant cosheaves. The properties of any quotient cosheaf in general can be derived in a similar manner (by the universal property of quotients). This is critical because the problem of understanding quotient spaces (here reciprocal diagrams) is translated into an equivalent problem of understanding its precursors (here self-stresses and ambient space) which are often much more tractable.

\begin{figure}[htp]
     \centering
     \hfill
     \begin{subfigure}[b]{0.28\textwidth}
         \centering
         \includegraphics[width=\textwidth]{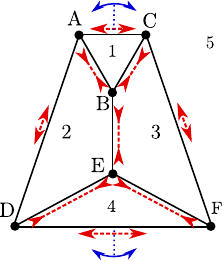}
         \caption{}
     \end{subfigure}
     \hfill
     \begin{subfigure}[b]{0.45\textwidth}
         \centering
         \includegraphics[width=\textwidth]{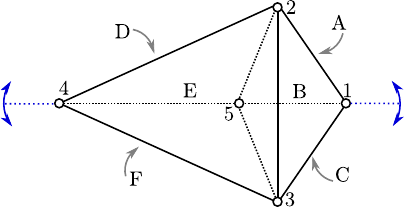}
         \caption{}
     \end{subfigure}
     \hfill
    \caption{A framework in {\it Desargues configuration} with a self-stress and vertical mirror-symmetry (a) and its corresponding parallel reciprocal diagram with horizontal mirror-symmetry (b). The ``transformation'' of the mirror from vertical to horizontal is a consequence of Remark~\ref{rem:sign_flip}.}
    \label{fig:mirror_gs}
\end{figure}

%%%%%%%%%%%%%%%%%%%%%%%%%%%%%%%%%%%%%%%%%%%%%%%%%%
\section{Finite Group Representations}
\label{sec:reps}
%%%%%%%%%%%%%%%%%%%%%%%%%%%%%%%%%%%%%%%%%%%%%%%%%%

A group representation is a homomorphism $\rho: G \to \text{GL}_k(W)$ where $G$ is a group and $W$ is a vector space over some field $k$. In this paper we will require the field $k$ to be $\R$ or $\C$ and the group $G$ to be finite. The dimension of $W$ is the \emph{dimension} of $\rho$.  We say $(W, \rho)$ is a {\it $(k)G$-module} under the group action $g\cdot w := \rho(g)w$ transforming vectors $w\in W$.

\begin{example}[Standard representation of cyclic and dihedral groups]
\label{ex:ZmDm_rep}
    Let $\Z_m$ be the cyclic group on $m$ elements and $D_{2m}$ be the dihedral group of order $2m$. (Note that $\Z_2=D_2$.) These groups act on the plane $\R^2$ by rotations and reflections by a two-dimensional representation $\tau$. Picking a rotation generator $r_1$, $\tau(r_1)\in SO(2)$ is a rotation matrix by angle $2\pi/m$. For a reflection $s$, $\tau(s)$ is a determinant $-1$ matrix with eigenvectors parallel or perpendicular to the line of reflection. 
\end{example}

We say two representations $\rho_0$ and $\rho_1$ are {\it equivalent} if there exists an invertable matrix ${\bf P}$ such that 
\begin{equation}\label{eq:rep_bases}
    \rho_0 = {\bf P}^{-1}\rho_1{\bf P}
\end{equation}
where we regard ${\bf P}$ as a change in basis. In coordinates, if $\scrb_0$ is a set of basis vectors for $V_0$ and $\scrb_1$ is a basis for $V_1$, then there is an invertable matrix ${\bf P}$ such that for each basis vector $b_0\in \scrb_0$ and $b_1\in \scrb_1$, $b_1 = {\bf P}b_0$.

Suppose that $(V, \rho)$ and $(W, \eta)$ are $G$-modules. A {\it $G$-homomorphism} $\phi: (V, \rho) \to (W, \eta)$ is a linear map $\phi:V \to W$ satisfying the natural equality $\eta(g)\circ \phi = \phi \circ \rho(g)$ for every $g\in G$. We say that $(V, \rho)$ and $(W, \eta)$ are isomorphic $G$-modules if there exists a $G$-isomorphism between them.

Suppose $W$ is a subspace of a $G$-module $(V, \rho)$. We say that $W$ is a ($\rho$-invariant) {\it $G$-submodule} if $\rho(g)w\in W$ for every $w\in W$. Then $\rho$ induces a representation $\rho_W: G\to \text{GL}_k(W)$ consisting of the restriction of $\rho$ to the subspace. The resulting embedding $(W, \rho_W) \to (V, \rho)$ is a $G$-homomorphism.

If $(W, \rho_W)$ is a $G$-submodule of $(V, \rho)$ we construct the $G$-quotient space $(V/W, \rho/ \rho_W)$. The group $G$ acts on a quotient vector $x + W$ by
\begin{equation}\label{eq:quotient_rep}
    \rho/\rho_W(g)(x+W) = \rho(g)(x) + \rho(g)W = \rho(g)(x) + W.
\end{equation}
Because the field $k$ is $\R$ or $\C$, by Maschke's Theorem there is a $G$-submodule $(U, \eta)$ such that $(V, \rho)$ is isomorphic to $(W \oplus U, \rho_W \oplus \eta)$ \cite{jl}.

A $G$-module $(V, \rho)$ is {\it irreducible} (and $\rho$ is irreducible) if its only $G$-submodules are zero and $(V, \rho)$ itself. Otherwise we say $V$ (or $\rho)$ is reducible. If $(W_0, \rho_1)$ and $(W_1, \rho_2)$ are $G$-submodules of $V$ such that $V\iso W_0 \oplus W_1$, then the representation $\rho$ is equivalent to the direct sum
\begin{equation*}
    (\rho_0 \oplus \rho_1)(g) =
    \begin{bmatrix}
    \rho_0(g) & 0 \\
    0 & \rho_1(g)
    \end{bmatrix}_{\scrb_0 \cup \scrb_1}
\end{equation*}
in the basis $\scrb_0 \cup \scrb_1$ of $W_0 \oplus W_1$, the ordered union of basis sets $\scrb_0$ for $W_0$ and $\scrb_1$ for $W_1$. Clearly the representations $(V, \rho)$ and $(W_1\oplus W_2, \rho_0\oplus \rho_1)$ are isomorphic.

\begin{example}[Irreducible representations of common cyclic and dihedral groups]\label{ex:irrep_tau}
The group $\Z_2$ has two $1$-dimensional irreducible representations, namely the one that assigns $1$ to both group elements, denoted by $\mu^{(1)}$, and the one that assigns $1$ to the trivial and $-1$ to the non-trivial group element, denoted by $\mu^{(2)}$. It is easy to see that the standard representation $\tau$ over $\Z_2$ from Example~\ref{ex:ZmDm_rep} decomposes as $\tau=\mu^{(1)}\oplus\mu^{(2)}$ in the case of reflection symmetry and $\tau=2\mu^{(2)}$ in the case of half-turn symmetry (see also \cite{schtan}).

Over the complex numbers, the cyclic group $\Z_m=\{0,\ldots, m-1\}$  has $m$ 1-dimensional irreducible representations $\mu^{(1)},\ldots,\mu^{(m)}$, where for
 each $j=1,\ldots, m$ and each $t\in \Z_m$, we have $\mu^{(j)}(t)=\zeta^{t(j-1)}$, with $\zeta=e^{\frac{2\pi i}{m}}$. 
A straightforward calculation shows that for $m\geq 3$, the standard representation $\tau$ of $\Z_m$ from Example~\ref{ex:ZmDm_rep} decomposes as $\tau=\mu^{(2)}\oplus\mu^{(m)}$ (see e.g. \cite{schtan,atk70}).

The dihedral group $D_4=\Z_2\times \Z_2=\{(0,0),(1,0),(0,1),(1,1)\}$ has four 1-dimensional irreducible representations $\mu^{(00)}, \mu^{(01)},  \mu^{(10)}, \mu^{(11)}$, 
which are defined by 
$\mu^{(j_1j_2)}((t_1,t_2))=(-1)^{j_1t_1+j_2t_2}$ for $0\leq j_1,j_2\leq 1$ and $(t_1,t_2)\in D_4$. It is again easy to see that the representation $\tau$ of $D_4$ decomposes as $\tau=\mu^{(10)}\oplus\mu^{(01)}$.

Finally, for all $m\geq 3$, the representation $\tau$ of $D_{2m}$ is an irreducible 2-dimensional representation  over the complex numbers \cite{atk70}.
\end{example}

The tensor product of two $G$-modules $(V_0, \rho_0)$ and $(V_1, \rho_1)$ is the $G$-module $(V_0 \otimes V_1, \rho_0 \otimes \rho_1)$ with the group action
\begin{equation*}
    (\rho_0 \otimes \rho_1) (g) (x_0 \otimes x_1) = (\rho_0(g) x_0) \otimes (\rho_1(g) x_1)
\end{equation*}

\begin{example}[Permutation Representation]
    Suppose $S$ is a finite set and $\alpha: G\times S \to S$ is a permutation action on $S$. Then $\alpha$ extends to a permutation representation on $k^{|S|}$, the vector space with a basis of formal elements of $S$. If $S=G$ and $\alpha$ is the group action acting by composition, then $\alpha$ is the regular representation.
\end{example}

With the field $k$ equal to $\R$ or $\C$ and the group $G$ finite, there are only finitely many irreducible representations of $G$ up to isomorphism \cite{jl}. We label these irreducible $G$-modules as $(k^{n_1}, \mu^{(1)} ), \dots, ( k^{n_m}, \mu^{(m)} )$ where $n_i$ is the dimension of the $i$-th irreducible representation and $m$ is the number of conjugacy classes of $G$. It is customary for the first irreducible $(k^{n_1}, \mu^{(1)})$ to be trivial, meaning $n_1 = 1$ and $\mu^{(1)}(g) = 1$ for every element $g\in G$. Every $G$-module can be uniquely decomposed as a direct sum of irreducible representations up to isomorphism \cite{jl}.

\begin{theorem}[\cite{jl}]\label{thm:Gkernel}
    Let $\phi:(V, \rho) \to (W, \eta)$ be a $G$-homomorphism. The subspaces $\ker\phi\subset V$ and $\im \phi \subset W$ are $G$-submodules under the action of $\rho$ and $\eta$. Likewise, the cokernel $W/\im \phi$ is a $G$-quotient space under $\eta$.
\end{theorem}

This theorem is utilized in the following fundamental lemma.

\begin{lemma}[(Partial) Schur's Lemma \cite{jl}
]\label{lem:schur}
    Suppose $(V, \rho)$ and $(W, \eta)$ are irreducible $G$-modules over the field $\R$ or $\C$. If $\phi:(V, \rho) \to (W, \eta)$ is a $G$-homomorphism then $\phi$ is an isomorphism or the zero map.
\end{lemma}

For each irreducible representation $\mu^{(j)}$ corresponding to the group $G$, let $( V^{(j)}, \rho^{(j)} )$ denote the $G$-submodule of $(V, \rho)$ isomorphic to the direct sum of all factors of the irreducible $G$-module $(k^{n_j}, \mu^{(j)} )$ in $(V, \rho)$. It follows that
\begin{equation}\label{eq:Gsubrep}
    (V, \rho) \iso (V^{(1)}, \rho^{(1)}) \oplus \dots \oplus (V^{(m)}, \rho^{(m)}).
\end{equation}

Suppose $V\iso V_0\oplus V_1$ and $W\iso W_0 \oplus W_1$ are two $G$-modules and constituent $G$-submodules. Further, suppose $\phi_0 :V_0 \to W_0$ and $\phi_1: V_1 \to W_1$ are two $G$-homomorphisms. We write $\phi_0 \oplus \phi_1: V_0 \oplus V_1 \to W_0 \oplus W_1$ for the combined $G$-homomorphism, represented by a block diagonal matrix
\begin{equation*}
    \phi_0 \oplus \phi_1 = 
    \begin{bmatrix}
        \phi_0 & 0 \\
        0 & \phi_1
    \end{bmatrix}
\end{equation*}
in some basis $\scrb_{V_0}\cup \scrb_{V_1}$ for the domain and a basis $\scrb_{W_0} \cup \scrb_{W_1}$ for the codomain.

We say an ordered basis $\scrb = (b_1, \dots, b_N)$ for a $G$-module $V$ is {\it adapted} if:
\begin{itemize}
    \item $b_\ell \in \scrb$ implies that $b_ \ell \in V^{(j)}$ for some $(j)$.
    \item If $b_\ell,b_{\ell'} \in \scrb$ with $b_\ell\in V^{(j)}$, $b_{\ell'} \in V^{(t)}$ and $j<t$, then $\ell < \ell'$. 
\end{itemize}
Letting $\scrb^{(j)}$ denote a basis for $V^{(j)}$, it follows that an adapted basis for $V$ is an ordered union of bases $\scrb^{(1)}\cup \dots \cup \scrb^{(m)}$. We say that a vector $x\in V$ is {\it $\mu^{(j)}$-symmetric} if $x$ is an element of the subspace $V^{(j)}$, or equivalently $x$ is a linear combination of basis vectors in $\scrb^{(j)}$.

\begin{theorem}[\cite{jl}]\label{thm:Ghomomorphism}
    Suppose $G$ is finite and the field $k$ is $\R$ or $\C$. Every $G$-homomorphism $\phi:(V, \rho) \to (W, \eta)$ decomposes as a direct sum of $G$-homomorphisms $\phi^{(1)} \oplus \dots \oplus \phi^{(m)}$ over irreducibles with
    \begin{equation*}
        \phi^{(j)}: V^{(j)} \to W^{(j)}.
    \end{equation*}
      In particular, the matrix $\phi^{(1)} \oplus \dots \oplus \phi^{(m)}$ is block diagonal with respect to adapted bases of $V$ and $W$.
\end{theorem}
\begin{proof}
    This is a direct consequence of Schur's Lemma~\ref{lem:schur}.
\end{proof}

\begin{theorem}
\label{thm:Gsequence}
    Fix a sequence of $G$-modules
    \begin{equation}\label{eq:Gsequence}
        \dots \to V_3 \xrightarrow{\phi_3} V_2 \xrightarrow{\phi_2} V_1 \xrightarrow{\phi_1} V_0 \xrightarrow{\phi_0} V_{-1} \to \dots.
    \end{equation}
    where maps $\phi_i$ are $G$-homomorphisms. For any irreducible representation $\mu^{(j)}$ of $G$, there is a well-defined sequence of $G$-submodules
    \begin{equation}\label{eq:split_sequence}
        \dots \to V_3^{(j)} \xrightarrow{\phi_3^{(j)}} V_2^{(j)} \xrightarrow{\phi_2^{(j)}} V_1^{(j)} \xrightarrow{\phi_1^{(j)}} V_0^{(j)} \xrightarrow{\phi_0^{(j)}} V_{-1}^{(j)} \to \dots.
    \end{equation}
    where $V_i \iso V_i^{(1)} \oplus \dots \oplus V_i^{(m)}$ are isomorphic $G$-modules and $\phi_i$ is equivalent to the map $\phi_i^{(1)} \oplus \dots 
    \oplus \phi_i^{(m)}$ in an adapted basis. Moreover, if sequence~\eqref{eq:Gsequence} is exact then sequence~\eqref{eq:split_sequence} is exact for every index $(j)$.
\end{theorem}

\begin{proof}
    By Theorem~\ref{thm:Ghomomorphism} each map $\phi_i$ is equivalent to the map $\phi_i^{(1)} \oplus \dots \oplus \phi_i^{(m)}$ in some adapted basis for $V_i$. Specifically the image of $\phi_i^{(j)}$ is contained in $V_{i-1}^{(j)}$, meaning that the composition $\phi_{i-1}^{(j)} \circ \phi_i^{(j)}$ is well defined.

    Suppose that sequence~\eqref{eq:Gsequence} is exact and fix an index $i$. For $j\neq t$ it must be true that $\ker \phi_i^{(j)}\cap \ker \phi_i^{(t)} = 0$. Thus we have equalities
    \begin{equation*}
        \ker \phi_i^{(1)} \oplus \dots \oplus \ker \phi_i^{(m)} = \ker \phi_{i} = \im \phi_{i+1} = \im \phi_{i+1}^{(1)} \oplus \dots \oplus \im \phi_{i+1}^{(m)}
    \end{equation*}
    of the $G$-submodule of $V_i$. It is also true for $j\neq t$ that $\ker \phi_i^{(j)} \cap \im \phi_i^{(t)} = 0$. Therefore we have an equality of $G$-submodules $\ker \phi_i^{(j)} = \im \phi_{i+1}^{(j)}$ for each $(j)$ and the sequence~\eqref{eq:split_sequence} is exact.
\end{proof}

We conclude this brief exposition on group representations with a comment on characters. The {\it character} of a (real or complex) representation $\rho$ is a class function $\chi_\rho : G \to \C$ given by the traces
\begin{equation*}
    \chi_\rho(g) = \text{trace}(\rho(g)).
\end{equation*}
The {\it degree} of a character $\chi_p$ is the dimension of the representation $\rho$. Characters are invariants of equivalent representations and hence are basis independent. It is useful to think of the character as a vector $\chi(\rho) \in \C^{|G|}$ where each coordinate is the trace of the matrix $\rho(g)$ (in some basis). Two group elements $g_0$ and $g_1$ are {\it conjugate} if there exists some $h\in G$ such that $g_0 = h^{-1}g_1h$. The traces of conjugate elements $g_0$ and $g_1$ are equal, and hence any {\it character table} only needs to list a representative from each conjugacy class (noting its multiplicity). Examples of character tables are shown in Figure~\ref{fig:star_recip} and Figure~\ref{fig:flower_recip}.

The following well-known identities then hold for any two representations $\rho_1$ and $\rho_2$:
\begin{align}
\label{eq:char_identities}
    \chi(\rho_0 \oplus \rho_1)(g) &= \chi(\rho_0)(g) + \chi(\rho_1)(g) \\
    \chi(\rho_0 \otimes\rho_1)(g) &= \chi(\rho_0)(g) \cdot \chi(\rho_1)(g).
\end{align}
Moreover, if $\rho_0$ and $\rho_1$ are two representations of a finite group $G$ then $\rho_0$ is equivalent to $\rho_1$ if and only if $\chi_0 = \chi_1$ \cite{jl}.

It is well known that the characters of the irreducible representations of a finite group $G$ form a basis for the {\it dual group} $\hat{G}$ of class functions $G\to \C$. \cite{jl}. This allows characters to be decomposed by the {\it character inner product}
\begin{equation}
    \langle \chi_0, \chi_1 \rangle = \frac{1}{|G|} \sum_{g\in G} \chi_0(g), \overline{\chi_1(g)}
\end{equation}
where $\overline{\chi_1(g)}$ is the complex conjugate of $\chi_1(g)$.

\begin{theorem}[Character Orthogonality \cite{jl}]\label{thm:char_orth}
    Over any finite group $G$ and any two irreducible representations $\mu^{(j)}$ and $\mu^{(t)}$, the following holds:
    \begin{equation}
        \langle \chi(\mu^{(j)}), \chi(\mu^{(t)}) \rangle = 
        \begin{cases}
            1 & \text{if} \; j = t \\
            0 & \text{if} \; j\neq t
        \end{cases}
    \end{equation}
    \end{theorem}

The characters $\chi(\mu^{(j)})$ from Theorem~\ref{thm:char_orth} are {\it irreducible characters} and can be looked up in standard references on representation theory \cite{alt94,atk70}. 

\begin{theorem}[Character Decomposition \cite{jl}]
\label{thm:char_decomp}
    Let $G$ be a finite group and $\chi$ be a character of $G$. Then $\chi$ can be uniquely decomposed into a linear combination of irreducible characters
    \begin{equation}
    \label{eq:char_decomp}
        \chi = \langle \chi, \chi(\mu^{(1)}) \rangle \chi(\mu^{(1)}) + \dots + \langle \chi, \chi(\mu^{(m)}) \rangle \chi(\mu^{(m)}).
    \end{equation}
\end{theorem}

These facts make character theory a essential tool for computing and decomposing into $G$-submodules.

%%%%%%%%%%%%%%%%%%%%%%%%%%%%%%%%%%%%%%%%%%%%%%%%%%
\section{Equivariant Cosheaves}
\label{equivariantcosheaf}
%%%%%%%%%%%%%%%%%%%%%%%%%%%%%%%%%%%%%%%%%%%%%%%%%%

The general theories of cellular cosheaves and finite group representations were reviewed independently in the previous two sections. In combination, we define equivariant cosheaves for describing symmetric data assignments. We then focus on applications: symmetric force loading assignments and symmetric reciprocal frameworks.

\begin{definition}[$G$-cell complex]
    For  a finite group $G$, a {\it $G$-cell complex} $(X, \alpha)$ is a cell complex $X$ with a permutation action $\alpha: G \times X \to X$ on the set of cells of $X$ satisfying:
    \begin{itemize}
        \item (Functorial) For any cell $c$, any $g,h\in G$, and $\epsilon\in G$ the identity element, $\alpha(g,\alpha(h,c)) = \alpha(gh,c)$ and $\alpha(\epsilon, c) = c$.
        \item (Equivariant) If $c\lhd d$ then $\alpha(g, c) \lhd \alpha(g, d)$.
    \end{itemize}
\end{definition}

\begin{definition}[$G$-cosheaf]\label{def:Gcosheaf}
    Suppose that $G$ is a finite group, $(X, \alpha)$ is a $G$-cell complex and $\calk$ is a $k$-cosheaf over $X$ such that $\calk_{c} \iso \calk_{gc}$ for every $g\in G$ and cell $c$. A {\it cosheaf representation} $\rho$ is a family of group representations on each space of chains $\{\rho_i: G \to \text{GL}_k(C_i \calk)\}$ such that:
    \begin{itemize}
        \item[(i)] For every $g\in G$, $\rho_{i-1}(g)\circ \bdd_i = \bdd_i \circ \rho_i(g)$.
        \item[(ii)] For $x$ an $i$-chain, the value of $\rho_i(g)(x)$ at a cell $gc$ depends only on $x_c$. In other words, there are isomorphisms for each cell $\rho_c:G \to \text{Orbit} (\calk_c)$ such that $(\rho_i(g)(x))_{gc} = \rho_c(g)x_c$.
    \end{itemize}
    We say the pair $(\calk, \rho)$ is a {\it $(k)G$-cosheaf}. 
\end{definition}

There are numerous observations we can make about $G$-cosheaves. From point (i), for each $g$, $\rho(g)$ is a $G$-chain complex\footnote{A $G$-chain complex is a functor from $G$ as a single object category to the category of chain complexes.} automorphism from $C \calk$ to itself. This means each chain space $C_i \calk$ is a $G$-module and the boundary maps $\bdd_i$ are $G$-homomorphisms. Utilizing point (ii) of Definition~\ref{def:Gcosheaf} and looking at the component of the boundary map corresponding to the incidence $c\lhd d$, we find that point (i) is equivalent to the equation
\begin{equation}\label{eq:Gchain_complex}
    \begin{aligned}
        (\rho_{i-1}(g)\circ \bdd_i(x_d))_{gc}
        & = [c:d]\rho_c(g)\calk_{d\rhd c} x_d \\
        &= [gc:gd]\calk_{gd\rhd gc} \rho_d(g) x_d \\
        & = (\bdd_i \circ \rho_i(g)(x_d))_{gc}.
    \end{aligned}
\end{equation}
being satisfied everywhere. Consequently, the ``local'' components of the cosheaf representation $\rho$, $\{\rho_c\}$ satisfy the commutative diagram
\begin{equation}\label{eq:Gcosheaf_failure}
    \begin{tikzcd}
        \calk_{d} \ar[r, "\rho_d(g)"] \ar[d, "\calk_{d\rhd c}"'] & \calk_{gd} \ar[d, "\calk_{gd \rhd gc}"] \\
        \calk_{c} \ar[r, "\rho_c(g)"] & \calk_{gc}
    \end{tikzcd}
\end{equation}
 up to sign, for every $g\in G$ and every incidence $c\lhd d$. Recalling diagram~\eqref{eq:Gcosheaf_map}, the above diagram~\eqref{eq:Gcosheaf_failure} is exactly the condition that $\rho(g)$ is a $G$-cosheaf map. Therefore, every $G$-cosheaf representation $\rho$ is nearly a $G$-natural cosheaf automorphism\footnote{We mean a functor from $G$ as a single object category to the category of natural transformations of cosheaves to themselves.}, and is so up to sign (this is a point worthy of future investigation).

\begin{example}[Trivial constant $G$-cosheaves]
\label{ex:trivial_Gconstant}
    Here we illustrate why a cosheaf representation may not be a $G$-indexed family of cosheaf maps. In particular, permuting the underlying cell interferes with preservation of orientation and signs, even on the simplest cosheaves.
    
    Suppose we have a $G$-cell complex $(X, \alpha)$ and $\overline{\R^n}$ is a constant cosheaf. The {\it trivial cosheaf representation} $\iota$ over $\overline{\R^n}$ is comprised of local maps $\iota_c(g) = [gc,c] \cdot \id$ for every $g\in G$ and cell $c$, where $[gc,c]\in \pm{1}$ measures orientation alignment between $c$ and $gc$. We let $[gc,c]=+1$ if the orientation of $c$ is reflected by the $g$ action, or $[gc,c]=-1$ if the orientation of $c$ is reversed. To satisfy equation~\eqref{eq:Gchain_complex} we require
    \begin{equation}\label{eq:sign_commute}
        [gc,c][c:d] = [gc:gd][gd, d]
    \end{equation}
    to hold true for all $c\lhd d$ and $g$. Requiring vertices to always have positive sign, $[gv,v]=+1$ for all $v$, an edge $u,v\lhd e$ changes sign and has $[ge,e]=-1$ when $[v:e][gv:ge]=-1$ (or equivalently $[u:e][gu:ge]=-1$). This sign change is demonstrated in Figure~\ref{fig:trivial_1}.
    
    With this trivial cosheaf action $\iota$, structure follows from the underlying permutation action $\alpha$. Over the unit constant cosheaf $\overline{\R}$, each map $\iota_i(g): C_i \overline{\R} \to C_i \overline{\R}$ is a signed permutation matrix with cells for basis elements. Abusing notation and declaring $\alpha_i$ to be a representation consisting of permutation matrices on $i$-cells in $C_i X = C_i \overline{\R}$, it follows that $\iota_0 = \alpha_0$ and $\iota_i$ is equivalent to $\alpha_i$ up to sign for $i>0$.
    
    Over $\overline{\R^n}$ with the trivial cosheaf action, each map $\iota_i(g):C_i \overline{\R^n} \to C_i \overline{\R^n}$ can be represented as a matrix with $\pm {\bf I}_n$ signed identity blocks. This is equivalent to the representation ${\bf I}_n\otimes \alpha_i$ on $C_i \overline{\R^n}$ up to sign, where ${\bf I}_n(g) = {\bf I}_n$ is the trivial representation on $\R^n$.
\end{example}

\begin{figure}
    \centering
    \includegraphics[scale = 1.4] {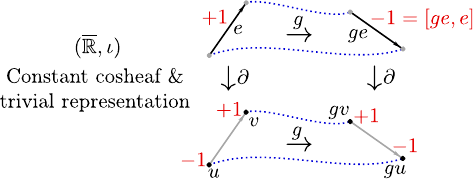}
    \caption{A sketch of the trivial representation $\iota$ over the constant cosheaf $\overline{\R}$, satisfying the constraint equation~\eqref{eq:Gchain_complex}. The edge $e$ changes orientation under the group action $g$, meaning $[ge,e] = -1$. The trivial representation $\iota$ does not detect cell geometry or embedding, only orientations.}
    \label{fig:trivial_1}
\end{figure}

\begin{example}[Regular representation]
    The group $G$ can be considered as a discrete cell complex comprised of a point for each group element. Let $(G, \ell)$ be a $G$-cell complex where $\ell$ is the left action of $G$ on itself. Suppose $(\overline{\R}, \iota)$ is the unit constant cosheaf over $(G,\ell)$ with the trivial $G$-action from the above example. Then the space of 0 chains $C_0(\overline{\R}, \iota)\iso \R^{|G|}$ is generated by a basis of group elements and $\iota = \iota_0 = \alpha_0$ is the left regular representation.
\end{example}

\begin{example}[Standard cyclic and dihedral constant cosheaf]
\label{ex:ZmDm_constant}
    When $G$ is a cyclic or dihedral group, there is a more useful representation than the trivial representation over constant cosheaves. We let $\eta$ be a cosheaf representation over $\overline{\R^2}$ determined by local maps $\eta_c(g) = [gc,c]\tau(g)$ where $\tau$ is the representation on $\R^2$ introduced in Example~\ref{ex:ZmDm_rep}. The representation $\eta_i(g)$ is equal to the representation $\tau \otimes \alpha_i$ when $i=0$ and is equivalent up to sign when $i>0$.
\end{example}

\begin{figure}
    \centering
    \includegraphics[scale = 1.4] {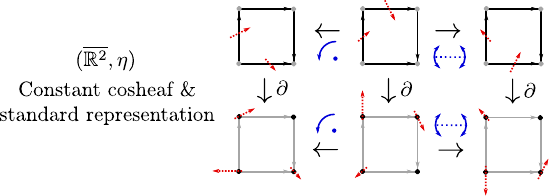}
    \caption{The $D_8$-constant cosheaf $(\overline{\R^2}, \eta_1)$ is pictured over a square cell complex, with 1 and 0 dimensional data drawn. Here we examine the commutativity condition (i) of Definition~\ref{def:Gcosheaf} over edges and vertices. To the left the group element is a $\pi/2$ rotation counter-clockwise, and to the right the group element is a reflection about the horizontal axis. Take note of the sign alignment $[ge,e]=\pm 1$ between an edge $e$ and its permutation.}
    \label{fig:eta_1}
\end{figure}

A realization $p:V \to \R^2$ is a {\it $G$-realization} if there is a representation $\tau_0: G \to \text{GL}(\R^2)$ over which $p$ is equivariant. In other words, $p$ satisfies
\begin{equation}\label{eq:Grealization}
    \tau_0(g) p_v = p_{gv}
\end{equation}
for every $g\in G$ and vertex $v\in V$. We will take $\tau_0 = \tau$ to be the standard representation from Example~\ref{ex:ZmDm_rep}. A {\it $G$-framework} is a $G$-cell complex $(X, \alpha)$ together with a $G$-realization forming a triple $(X, \alpha, p)$.

\begin{example}[Cyclic and dihedral force cosheaf]
\label{ex:ZmDm_force}

    We investigate the force cosheaf $\calf$ over such a realized $G$-cell complex $(X, \alpha, p)$. Due to the isomorphism of vertex stalks $\calf_v \iso \overline{\R^2}$, we can consider the representation $\rho$ on $\calf$ extending the representation $\tau \otimes \alpha_0$ on $C_0 \calf \iso C_0 \overline{\R^2}$ to 1-chains. We set $\rho_e(g) = \id$ between edge stalks $\calf_e \to \calf_{ge}$ sending a unit compression over $e$ to unit compression over $ge$ for every edge $e$ and element $g\in G$. With this identification $\rho_1$ is equivalent to $\alpha_1$, the representation of (strictly positive) permutation matrices on edge generators in $C_1\calf \iso C_1 X$. To confirm that $\rho$ is indeed a $G$-representation we check the condition~\eqref{eq:Gchain_complex}, namely the the following vectors are equal
    \begin{equation*}
        \begin{aligned}
            \relax [gv:ge]\calf_{ge\rhd gv} \circ \id(1)
            & = [gv:ge]^2(p_{gv} - p_{gu})
            \\
            & = [v:e]^2\tau(g)(p_{v} - p_{u})
            \\
            & = [v:e]\tau(g) \circ \calf_{e\rhd v}(1).
        \end{aligned}
    \end{equation*}
    Here we utilized the definitions of a $G$-realization~\eqref{eq:Grealization} and of the force cosheaf extension map in line~\eqref{eq:force_extension}.
    
    The force cosheaf $\calf$ and its cosheaf representation $\rho$ are used in most other sources of equivariant trusses \cite{connelly_guest_2022,SCHULZE2023112492}.
\end{example}

\begin{figure}
    \centering
    \includegraphics[scale = 1.4] {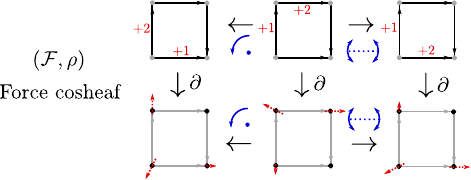}
    \caption{The $D_8$-force cosheaf $(\calf, \rho)$ is pictured over a square. Two of the edges are in varying degrees of compression, and we check condition (i) of Definition~\ref{def:Gcosheaf} for $\calf$.}
    \label{fig:rho_1}
\end{figure}

Recall that $\mu^{(1)}, \dots , \mu^{(m)}$ denotes the irreducible representations of $G$ unique up to isomorphism. For each dimension index $i$ we have that $C_i\calk, \rho$ is isomorphic to a direct sum of $G$-submodules
\begin{equation*}
    C_i\calk \iso C_i^{(1)}\calk \oplus C_i^{(2)}\calk \oplus \dots \oplus C_i^{(m)}\calk
\end{equation*}
where $C_i^{(j)}\calk$ is isomorphic to the direct sum of $d_i^{(j)}\geq 0$ copies of the irreducible $G$-module $(k^{n_j}, \mu^{(j)})$. Each space $C_i^{(j)}\calk$ has dimension $d_i^{(j)}n_j$.

Let $\bdd_i^{(j)}$ denote the restriction of the boundary map $\bdd_i$ to the subspace $C_i^{(j)}\calk$. By Theorem~\ref{thm:Ghomomorphism} it follows that $\bdd_i$ can be represented by the sum
\begin{equation}\label{eq:split_bdd}
    \bdd_i = \bdd_i^{(1)} \oplus \dots \oplus \bdd_i^{(m)}
\end{equation}
over the irreducible $G$-submodules. The image of $\bdd_i^{(j)}$ is zero on $C_{i-1}^{(t)}$ for any $t\neq j$.
\begin{lemma}
    Any $G$-cosheaf chain complex $C(\calk, \rho)$ is isomorphic to the decomposition
\begin{equation}\label{eq:chain_irrid}
    C \calk \iso C^{(1)}\calk \oplus \dots \oplus C^{(m)}\calk
\end{equation}
where each chain complex $C^{(j)}\calk$ is of the form
\begin{equation*}
    \dots \to C_{i+1}^{(j)}\calk \xrightarrow{ \bdd_{i+1}^{(j)}} C_i^{(j)}\calk \xrightarrow{ \bdd_i^{(j)} } C_{i-1}^{(j)}\calk \to \dots
\end{equation*}
\end{lemma}
\begin{proof}
    This is Theorem~\ref{thm:Gsequence} applied to the boundary maps $\bdd_i$ of the chain complex.
\end{proof}

%%%%%%%%%%%%%%%%%%%%%%%%%%%%%%%%%%%%%%%%%%%%%%%%%%
\subsection{Equivariant Cosheaf Maps}
%%%%%%%%%%%%%%%%%%%%%%%%%%%%%%%%%%%%%%%%%%%%%%%%%%

\begin{definition}[$G$-cosheaf map]\label{def:Gchain_map}
    Suppose $G$ is a finite group, $(X,\alpha)$ is a $G$-cell complex, $(\calk, \rho)$ and $(\call, \eta)$ are $G$-cosheaves and $\phi:\calk \to \call$ is a cosheaf map. We say that $\phi$ is a {\it $G$-cosheaf map} if it satisfies the commutative diagram
    \begin{equation}\label{eq:chain_cube}
        \begin{tikzcd}[column sep={between origins,5em}]
            C_i\calk \ar[rr, "\phi"] \ar[dd, "\bdd"] \ar[rd, "\rho(g)"] 
            & & C_i\call \ar[dd, "\bdd", near end] \ar[rd, "\eta(g)"] & \\
            & C_i\calk \ar[rr, "\phi", crossing over, near start] 
            & & C_i\call \ar[dd, "\bdd"] \\
            C_{i-1}\calk \ar[rr, "\phi", near start] \ar[rd, "\rho(g)", swap] & & C_{i-1}\call \ar[rd, "\eta(g)"] & \\
            & C_{i-1}\calk \ar[rr, "\phi"] \ar[from=uu, "\bdd", crossing over, near end]
            & & C_{i-1}\call
        \end{tikzcd}
    \end{equation}
    on chains for every index $i$ and group element $g\in G$. The composition of maps over every path from $C_i\calk$ to $C_{i-1}\call$ must be equal.
\end{definition}

The commutativity of the front and back squares of diagram~\eqref{eq:chain_cube} follow from $\phi$ being a chain map and the commutativity of the left and right squares follow by both $\calk$ and $\call$ being $G$-cosheaves (by assumption). The only statements that must be checked are the commutativity of the top and bottom squares of diagram~\eqref{eq:chain_cube}. In particular, for every $i$-chain $x\in C_i \calk$ it must be true that
\begin{equation}\label{eq:Gcosheaf_rep_map}
    \phi_{gc} \rho_c(g) x_c = \eta_c(g) \phi_c x_c.
\end{equation}

Every $G$-cosheaf map $\phi$ consists of a family of $G$-homomorphisms on chain spaces $\{\phi: C_i \calk \to C_i \call\}$.

A short exact sequence of {\it $G$-cosheaf maps}
\begin{equation}\label{eq:Gchain_ses}
    0 \to C(\calk, \rho) \xrightarrow{\phi} C(\call, \eta) \xrightarrow{\psi} C(\calm, \mu) \to 0
\end{equation}
is a short exact sequence of cosheaf maps, equivariant under the respective representation actions. From any injective $G$-cosheaf map $\phi: (\calk, \rho) \to (\call, \eta)$ the quotient $G$-cosheaf $(\call/\phi\calk, \eta/\phi\rho)$ has stalks $(\call/\phi\calk)_c = \call_c / \im \phi_c \calk_c$. The group action on stalks is the action $\eta$ on quotient classes:
\begin{equation*}
\begin{aligned}
    (\eta/\phi\rho)_c(g)(x+\im \phi_c)
    & = \eta_c(g)(x) + \im (\eta_c(g)\circ \phi_c)
    \\
    & = \eta_c(g)(x) + \im (\phi_{gc}\circ\rho_c(g) \phi_c)
    \\
    & = \eta_c(g)(x) + \im \phi_{gc}
\end{aligned}
\end{equation*}

We simplify the notation letting $\eta/\rho$ denote the representation $\eta/\phi \rho$.

\begin{example}[Cyclic and dihedral position cosheaf]
\label{ex:ZmDm_cosheaf_map}
    Letting $G$ be cyclic or dihedral, recall from Example~\ref{ex:ZmDm_force} and Example~\ref{ex:ZmDm_constant} that we defined the appropriate cosheaf representations for the force $\calf$ and constant $\overline{\R^2}$ cosheaves. In Example~\ref{ex:graphic_statics} we developed classical planar graphic statics and proved that the structure of the position cosheaf $\calp$ can be deduced from these previous two cosheaves (without $G$-action). We wish to do the same {\it while including the $G$-action}, namely we show that the appropriate representation of $\calp$ can be derived purely from the representions of $(\calf, \rho)$ and $(\overline{\R^2}, \eta)$. This is the subject of Lemma~\ref{lem:ZmDm_position}, and we describe the $G$-cosheaf $(\calp, \eta/\rho)$ for now.
    
    When assuming the underlying $G$-cosheaf is an oriented $2$-manifold, it is possible to assign every face a consistent local orientation. Then $[gf,f]= \pm 1$ depending on if $g$ is a rotation or a reflection. There is a simple formulation of $\eta/\rho$, namely ${\eta/ \rho}_2 = \eta_2$ consisting of local maps
    \begin{equation}\label{eq:position_rep_face}
        {\eta/\rho}_f(g) = \begin{cases}
            +\tau(g) & g\text{ is a rotation} \\
            -\tau(g) & g\text{ is a reflection}
        \end{cases},
    \end{equation}
    which is always positive when the group $G$ is $\Z_m$. Over edges the representation $\eta / \rho_1$ is similar
    \begin{equation}\label{eq:position_rep_edge}
        {\eta / \rho}_e(g) = \begin{cases}
            +1 & g\text{ is a rotation}\\
            -1 & g\text{ is a reflection}
        \end{cases}.
    \end{equation}
\end{example}

We emphasize that it is important to deduce properties of the quotient $G$-cosheaf $(\calp, \eta/\rho)$ purely in terms of its priors $(\calf, \rho)$ and $(\overline{\R^2}, \eta)$. Universal properties (here, of quotients) are extremely powerful, and  properties of any quotient $G$-cosheaf can be derived in an algorithmic manner by diagram chasing. The authors hope the methods used in Lemma~\ref{lem:ZmDm_position} guide future derivations of quotient $G$-actions.

\begin{remark}\label{rem:sign_flip}
The sign flip in Equations~\eqref{eq:position_rep_face} and~\eqref{eq:position_rep_edge} has counter-intuitive effects. Take the reflection $s$ (in some dihedral group) along the vertical axis; the standard representation $\tau$ takes value
\begin{equation*}
    \tau(s) = \begin{small}
    \begin{bmatrix}
        -1 & 0 \\
        0 & 1
    \end{bmatrix}
    = -\begin{bmatrix}
        1 & 0 \\
        0 & -1
    \end{bmatrix}\end{small},
\end{equation*}
inverting the first coordinate. Then at any face $f$, $\eta/\rho_f(s)$ takes the value $-\tau(s)$, which is a matrix that inverts the second coordinate, ``acting like'' a reflection along the horizontal axis. We can think of $\eta/\rho$ having dual mirror symmetry in the dihedral group, a phenomenon previously noticed in \cite{SCHULZE2023112492}. An example of this mirror-transformation is clearly visible in Figure~\ref{fig:mirror_gs}, also pictured in detail in Figure~\ref{fig:eta_2}.
\end{remark}

\begin{lemma}\label{lem:ZmDm_position}
    The position cosheaf $(\calp, \eta/\rho)$ defined in the above Example~\ref{ex:ZmDm_cosheaf_map} indeed is a $\Z_n$- or $D_n$-cosheaf and $\pi:(\overline{\R^2}, \eta) \to (\calp, \eta/\rho)$ is a $\Z_n$- or $D_n$-cosheaf map.
\end{lemma}
\begin{proof}
    Fix the group $G$ as $\Z_m$ or $D_{2m}$. 
    We first prove the map $\phi:\calf \to \overline{\R^2}$ is in-fact a $G$-cosheaf map satisfying the condition of line~\eqref{eq:Gcosheaf_rep_map} at all cells and group elements. Let $(\calf,\rho)$ and $(\overline{\R^2},\eta)$ be the force and constant $G$-cosheaves of Examples~\ref{ex:ZmDm_constant} and~\ref{ex:ZmDm_force}. Line~\eqref{eq:Gcosheaf_rep_map} is satisfied trivially over vertices, as both $\rho_0$ and $\eta_0$ are equivalent to the same representation $\tau\otimes\alpha_0$. For commutativity over edge cells, we note that $\rho_e(g) = +1$ and that $\eta_e(g) = [ge,e]\tau(g)$. Then the orientation of the edge $e$ being preserved/reversed to $ge$ is equivalent to the base of $e$ (the vector $(p_u - p_v)$ for $u,v\lhd e$) being preserved/reflected, and
    \begin{equation}\label{eq:phi_commute}
        \phi_{ge} = [ge,e]\tau(g)\phi_e
    \end{equation}

    Because $\phi$ is a $G$-cosheaf map, the image $\phi C \calf$ is a $G$-submodule of the chain complex $C \overline{\R^2}$. We next prove that that quotient map $\pi$ from Example~\ref{ex:graphic_statics} is also a well-defined $G$-cosheaf map, and in the process show that line~\eqref{eq:position_rep_face} and line~\eqref{eq:position_rep_edge} hold true.
    
    We check the commutativity of the diagram~\eqref{eq:chain_cube} for the map $\pi$ over $2$ and $1$ chains. By assumption $\pi$ is a (regular) cosheaf map, meaning the front and back squares of the diagram commute. Also the left square commutes by the construction of $(\overline{\R^2}, \eta)$ in Example~\ref{ex:ZmDm_constant}. Clearly $\pi \circ \eta_2(g) = \eta/\rho_2(g) \circ \pi$ as maps from $C_2 \overline{\R^2}$ to $C_2 \calp$ so the top square commutes.
    
    We show that $\pi \circ \eta_1(g) = (\eta/\rho)_1(g) \circ \pi$, verifying the commutativity of the bottom square of diagram~\eqref{eq:chain_cube}. For computations, we associate $\calp_e$ over an edge $u,v\lhd e$  with the orthogonal space $(\phi \calf_e)^\perp$ in $\R^2$ by rotating the unit vector $(p_u - p_v) = \phi_e (1)$ generating $\phi \calf_e$ by angle $\pi/2$ clockwise
    \begin{equation*}
        \calp_e \iso \text{span}\{{\bf R}(\pi/2) \phi_e (1)\}.
    \end{equation*}
    where ${\bf R}(\cdot)$ is the rotation matrix by the specified angle. Then we define $\pi_1: C_1\overline{\R^2} \to C_1\calp$ by setting $\pi_e(y_e) = \langle {\bf R}(\pi/2) \phi_e(1), y_e \rangle$ for $y\in C_1 \overline{\R^2}$. We know that at an edge $e$ and $g\in G$, $\eta/\rho_e(g)$ is a scalar, thus we have
    \begin{equation}\label{eq:pos_rep_computation}
        \begin{aligned}
            \pi_{ge}\circ \eta_e(g) (y_e)
            & = {\eta/\rho}_e \circ \pi_e (y_e) \\
            [ge,e]\langle {\bf R}(\pi/2) \phi_{ge}(1), \tau(g) y_e \rangle
            & = {\eta/\rho}_e(g)\langle {\bf R}(\pi/2) \phi_e(1), y_e \rangle
            \\
            [ge,e]^2 \tau(g)^{-1} {\bf R}(\pi/2) \tau(g) \phi_e(1)
            & = {\eta/\rho}_e(g){\bf R}(\pi/2) \phi_e(1)
        \end{aligned}
    \end{equation}
    using line~\eqref{eq:phi_commute}. This implies that
    \begin{equation*}
        {\eta/\rho}_e(g) \phi_e(1)
        = {\bf R}(\pi/2)^{-1} \tau(g)^{-1} {\bf R}(\pi/2)\tau(g)\phi_e(1)
        = [{\bf R}(\pi/2), \tau(g)]\phi_e(1)
    \end{equation*}
    the commutator of the two orthogonal matrices. If $g$ is a rotation, the matrices commute and ${\eta/\rho}_e = +1$. When $g$ is a reflection then
    \begin{equation*}
        {\bf R}(\pi/2)^{-1} \left(\tau(g)^{-1} {\bf R}(\pi/2) \tau(g) \right)
        = {\bf R}(\pi/2)^{-1} {\bf R}(\pi/2)^{-1}
        = {\bf R}(\pi)
    \end{equation*}
    sends $\phi_e(1)$ to $-\phi_e(1)$.

    For the final right commutativity square of diagram~\eqref{eq:chain_cube}, we confirm that $(\calp, \eta/\rho)$ is indeed a $G$-cosheaf by checking point (i) of Definition~\ref{def:Gcosheaf}. For $x\in C_2 \calp$ the following equations are equivalent,
    \begin{equation}\label{eq:right_square}
    \begin{aligned}
        \relax [ge:gf] \calp_{gf \rhd ge} \eta/\rho_f(g) x_f
        & = {\eta/\rho}_e(g) [e:f] \calp_{f\rhd e} x_f
        \\
        [ge:gf][gf,f] \langle {\bf R}(\pi/2)\phi_{ge}(1), \tau(g) x_f \rangle
        & = {\eta/\rho}_e(g) [e:f]\langle {\bf R}(\pi/2) \phi_e (1), x_f \rangle,
    \end{aligned}
    \end{equation}
    which after using equation~\eqref{eq:sign_commute}, ~\eqref{eq:right_square} is identical to line~\eqref{eq:pos_rep_computation} swapping variables $x$ and $y$. Thus $\pi$ is a $G$-map between $G$-cosheaves.
\end{proof}

\begin{figure}
    \centering
    \includegraphics[scale = 1.4]{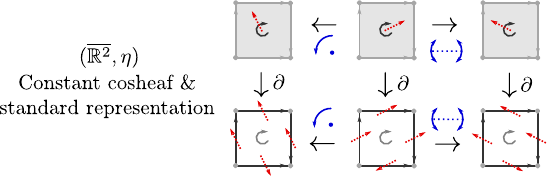}
    \caption{Generators for the representation of the $D_8$-constant cosheaf $(\overline{\R^2}, \eta_2)$ are pictured acting on a square face (two-dimensional data). This is a continuation of Figure~\ref{fig:eta_1}. To the left, the stalks are permuted, rotated by $\pi/2$ counter-clockwise, then resigned by multiplying by the scalar $[gc,c]$. To the right, the stalks are permuted, reflected along the horizontal axis, then resigned. The composition of reflection along the horizontal axis and negation misleadingly appears like a reflection along the vertical axis, following Remark~\ref{rem:sign_flip}.}
    \label{fig:eta_2}
\end{figure}

\begin{figure}
    \centering
    \includegraphics[scale = 1.4]{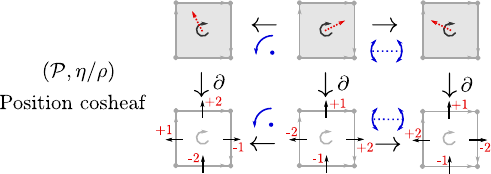}
    \caption{Generators for the representation of $(\calp, \eta/\rho)$ are pictured acting on a square face. Note that the top row is identical to that of Figure~\ref{fig:eta_2}, and the basis vectors orthogonal to edges are drawn in the bottom row. The vector in the top-center square has components $(+2, +1)$ in the $x$ and $y$ directions. To the left, vectors/scalars are permuted and rotated by $\pi/2$ counter-clockwise. To the right, the square is reflected along the horizontal axis, then some vectors/scalars reverse sign following equation~\eqref{eq:position_rep_face}. The  composition appears like a reflection along the vertical axis, following Remark~\ref{rem:sign_flip}.}
    \label{fig:eta_rho_2}
\end{figure}

As consequence of Example~\ref{ex:ZmDm_cosheaf_map} and Lemma~\ref{lem:ZmDm_position}, we know that 
\begin{equation}\label{eq:GS_chainses}
    0 \to C(\calf, \rho) \xrightarrow{\phi} C(\overline{\R^2}, \eta) \xrightarrow{\pi} C(\calp, \eta/\rho) \to 0.
\end{equation}
is an exact sequence of $\Z_m$- or $D_{2m}$-cosheaves.

%%%%%%%%%%%%%%%%%%%%%%%%%%%%%%%%%%%%%%%%%%%%%%%%%%
\section{Irreducible Representations and Homology}
\label{sec:irrid_homology}
%%%%%%%%%%%%%%%%%%%%%%%%%%%%%%%%%%%%%%%%%%%%%%%%%%

Previously we confirmed that cosheaves and maps between cosheaves can be enriched with group representations. From this groundwork we have the methods for separating cosheaf chains and homology cycles into their constituent irreducible components, each respecting one of the underlying symmetries of the framework.

\begin{lemma}\label{lem:split_complex}
    To any short exact sequence of cosheaf chain complexes of the form~\eqref{eq:Gchain_ses}, there is a short exact sequence of $G$-cosheaf chain complexes for each irreducible representation $\mu^{(j)}$ of $G$:
    \begin{equation}\label{eq:split_complex}
        0 \to C^{(j)}(\calk, \rho) \xrightarrow{\phi^{(j)} } C^{(j)}(\call, \eta) \xrightarrow{\psi^{(j)} } C^{(j)}(\call / \calk, \eta/\rho) \to 0
    \end{equation}
\end{lemma}
\begin{proof}
    Each component chain map $\phi_i:C_i\calk \to C_i \call$ and $\psi_i: C_i \call \to C_i \calm$ are $G$-homomorphisms. By Theorem~\ref{thm:Gsequence} the sequence
    \begin{equation}\label{eq:split_chains}
        0 \to C^{(j)}_i(\calk, \rho) \xrightarrow{\phi_i^{(j)} } C^{(j)}_i(\call, \eta) \xrightarrow{\psi_i^{(j)} } C^{(j)}_i(\call / \calk, \eta/\rho) \to 0.
    \end{equation}
    of $i$ chains is exact. We know $\phi_i$ and $\psi_i$ commute with the respective cosheaf boundary maps $\bdd$ which likewise decompose along the irreducible components in line~\eqref{eq:split_bdd}. Thus exact sequences of $G$-modules~\eqref{eq:split_chains} extend to exact sequences of $G$-chain complexes~\eqref{eq:split_complex}.
\end{proof}

The long exact sequence~\eqref{eq:les} respects the group action of the $G$-cosheaves. We have seen that representations $\rho(g): C\calk \to C\calk$ and $\eta(g): C\call \to C\call$ are chain complex automorphisms for each $g\in G$. These maps (as quasi-automorphisms) induce isomorphisms in homology $\rho(g): H\calk \to H \calk$ and $\eta: H \call \to H \call$. The following diagram with exact rows commutes for every $g\in G$:
\begin{equation}\label{eq:Gles}
\begin{tikzcd}
    \dots \ar[r, "\psi"] &
    H_{i+1} \call / \calk \ar[r, "\vartheta"] \ar[d, "\eta/\rho(g)"] &
    H_i \calk \ar[r, "\phi"] \ar[d, "\rho(g)"] &
    H_i\call \ar[r, "\psi"] \ar[d, "\eta(g)"] &
    H_i \call / \calk \ar[r, "\vartheta"] \ar[d, "\eta/\rho(g)"] &
    \dots \\
    \dots \ar[r, "\psi"] &
    H_{i+1} \call / \calk \ar[r, "\vartheta"] &
    H_i \calk \ar[r, "\phi"] &
    H_i\call \ar[r, "\psi"] &
    H_i \call / \calk \ar[r, "\vartheta"] &
    \dots \\
\end{tikzcd}
\end{equation}
following from the naturality of the long exact sequence \cite{AlgebraicHatcher2002}.

\begin{lemma}\label{lem:split_les}
    Every long exact sequence of $G$-cosheaf homology splits into irreducible factors. For a short exact sequence of $G$-cosheaves, to each irreducible representation $\mu^{(j)}$ of $G$ the sequence
    \begin{equation}\label{eq:split_les}
        \dots \to H_{i+1}^{(j)} \call / \calk \xrightarrow{\vartheta^{(j)}} H_i^{(j)} \calk \xrightarrow{\phi^{(j)}} H_i^{(j)} \call \xrightarrow{\psi^{(j)}} H_i^{(j)} \call / \calk \to \dots
    \end{equation}
    is exact.
\end{lemma}
\begin{proof}
Line~\eqref{eq:split_les} is the long exact sequence of the chain complex~\eqref{eq:split_complex}. By Theorem~\ref{thm:Gsequence}, the connecting homomorphism $\vartheta^{(j)}$ is exactly the $(j)$-th irreducible component of the full connecting $G$-homomorphism $\vartheta: H_{i+1} \call/\calk \to H_i \calk$ from diagram~\eqref{eq:Gles}.
\end{proof}

%%%%%%%%%%%%%%%%%%%%%%%%%%%%%%%%%%%%%%%%%%%%%%%%%%
\subsection{A Symmetrical Maxwell Counting Rule}
%%%%%%%%%%%%%%%%%%%%%%%%%%%%%%%%%%%%%%%%%%%%%%%%%%

Characters have useful orthogonality and projection properties \cite{jl} that make character theory a critical tool for counting the dimensions of chain and homology spaces. We demonstrate the utility of character theory by formulating a symmetric version of Maxwell's rule for frameworks in terms of a symmetric Euler characteristic for chain complexes using characters.

The {\it Euler characteristic} of a finite dimensional chain complex $C$ is the alternating sum of the dimensions
\begin{equation*}
    \calx(C) = \sum_i (-1)^i \dim C_i
\end{equation*}
The Euler formula has found use in molecular chemistry \cite{ceulemans1991extension}, DNA polyhedra \cite{hu2011new}, configuration spaces in robotics \cite{farber2008invitation}, and of course in structural mechanics \cite{connelly_guest_2022, cooperband2023homological} among many other applications.

\begin{theorem}[Standard Euler Characteristic \cite{AlgebraicHatcher2002}]\label{thm:Euler}
    The Euler characteristic of a finite dimensional chain complex and its homology are equal. In particular, for a cosheaf $\calk$ with finite dimensional stalks over a finite dimensional cell complex we have 
    \begin{equation*}
        \calx(C\calf) = \sum_{i} (-1)^i \dim C_i \calk = \sum_{i} (-1)^i \dim H_i\calk = \calx(H\calf).
    \end{equation*}
\end{theorem}
Theorem~\ref{thm:Euler} also applies to chain complexes corresponding to the irreducible representations of $G$, namely $\calx(C^{(j)}\calk) = \calx(H^{(j)} \calk)$.

The well-known {\it Maxwell counting rule} is the statement that the difference in dimensions of self stresses and kinematic degrees of freedom is equivalent to counting the different dimension cells of a finite framework \cite{Calladine1978}; in two dimensions this is the equation
\begin{equation}
\label{eq:maxwell_rule}
    \#\text{kinematic degrees of freedom} - \#\text{self stress dimensions} = 2\#\text{vertices} - \#\text{edges}
\end{equation}
where global translation and rotation assignments are included in the kinematic space. This is exactly the application of Theorem~\ref{thm:Euler} to the force cosheaf $\calf$ \cite{cooperband2023homological}.

Over a $G$-framework $(X, \alpha, p)$, the Maxwell counting rule~\eqref{eq:maxwell_rule} takes a more refined form. At the identity element $\epsilon$, the character $\chi_\rho(\epsilon) = \text{trace}(\rho(\epsilon))$ is nothing more than the degree of $\chi_\rho$. Therefore the Standard Euler Characteristic Theorem (Theorem~\ref{thm:Euler}) for a $G$-cosheaf $(\calk, \rho)$ can be reformulated as
\begin{equation}
\label{eq:euler_e}
    \calx(C\calk) = \sum_i (-1)^i\chi(\rho_i)(\epsilon) = \sum_i (-1)^i \chi(\rho_{H_i \calk})(\epsilon) = \calx(H \calk).
\end{equation}
However, there is no need to restrict the character $\chi(\rho)$ to only the identity element. Equation~\eqref{eq:euler_e} in vector form (in $\C^{|G|}$) is
\begin{equation}
\label{eq:char_euler}
    \sum_i (-1)^i \chi(\rho_i) = \sum_i (-1)^i \chi(\rho_{H_i \calk}),
\end{equation}
which is quickly seen from repeated application of Theorem~\ref{thm:Gkernel} and character identities~\eqref{eq:char_identities}. Equation~\eqref{eq:char_euler} applies to any $G$-cosheaf chain complex $(C\calk, \rho)$, including its irreducible component chain complexes $(C^{(j)}\calk, \rho)$. For any index $(j)$ the following identity holds:
\begin{equation}
\label{eq:irrid_euler}
    \sum_i (-1)^i \chi(\rho^{(j)}_i) = \sum_i (-1)^i \chi(\rho^{(j)}_{H_i \calk}).
\end{equation}

By Theorem~\ref{thm:char_decomp}, each representation $\rho_i$ is equivalent to a direct sum of an integer number $N_i^{(j)}(\rho_i) = \langle \chi(\rho_i), \chi(\mu^{(j)}) \rangle$ of factors of the irreducible representation $\mu^{(j)}$. In particular,
\begin{equation}
    \label{eq:char_counting}
    \chi(\rho_i) = \sum_{(j)} N_i^{(j)}(\rho) \cdot \chi(\mu^{(j)})
\end{equation}
is the decomposition by irreducible characters~\eqref{eq:char_decomp}.

\begin{definition}
    For a finite group $G$, the {\it symmetric Euler characteristic} of a finite dimensional $G$-chain complex $(C, \rho)$ is
    \begin{equation}
        \label{eq:Geuler_chain}
        \hat{\calx}(C, \rho) =
        \left(
        \sum_i (-1)^i N_i^{(1)}(\rho), \dots, \sum_i (-1)^i N_i^{(m)}(\rho)
        \right),
    \end{equation}
    an $m$-tuple of integers where $m$ is the number of irreducible representations of the group $G$.
\end{definition}

\begin{theorem}[Symmetric Euler characteristic]
    For a finite group $G$, the symmetric Euler chararacteristics of a finite $G$-chain complex and its homology are equal. In particular for a $G$-cosheaf $\calk$ with finite dimensional stalks over a finite dimensional cell complex, the $(j)$-th components of the symmetric Euler characteristics are equal to
    \begin{equation}
    \label{eq:Geuler_char}
        \hat{\calx}^{(j)}(C\calk, \rho) = \sum_i (-1)^i N_i^{(j)}(\rho) = \sum_i (-1)^i N_i^{(j)}(\rho_{H_i \calk}) = \hat{\calx}^{(j)}(H\calk, \rho_{H\calk})
    \end{equation}
    for every index ${(j)}$.
\end{theorem}
\begin{proof}
    Clearly the equalities
    \begin{equation*}
        \hat{\calx}^{(j)}(C\calk, \rho) \cdot \chi(\mu^{(j)})  = \sum_i (-1)^i N_i^{(j)}(\rho) \cdot \chi(\mu^{(j)}) = \sum_i (-1)^i \chi(\rho^{(j)}_i)
    \end{equation*}
    hold by the defining Equation~\eqref{eq:Geuler_chain}. A similar equality holds for homologies $\hat{\calx}^{(j)}(H\calk, \rho_{H\calk})$. The result then follows from Equation~\eqref{eq:irrid_euler}.
\end{proof}

Note that when $G$ is the trivial group, the extended Euler characteristic $\hat{\calx}(C\calk, \id)$ is nothing more than the standard Euler characteristic $\calx(C\calk)$. In fact, for any finite group the standard Euler characteristic can be recovered by taking the degree of both sides of Equation~\eqref{eq:char_counting}; this is the weighted sum
\begin{equation*}
    \calx(C \calk) = \sum_i (-1)^i \deg \chi(\rho_i) = \sum_i \sum_{(j)} (-1)^i N_i^{(j)}(\rho)\cdot \deg \chi(\mu^{(j)}) = \sum_{(j)} \hat{\calx}^{(j)} (C\calk, \rho) \cdot \dim \mu^{(j)}.
\end{equation*}

Over the $G$-force cosheaf $(\calf, \rho)$, Equation~\eqref{eq:Geuler_char} is the {\it symmetric Maxwell rule}, the analogue of Equation~\eqref{eq:maxwell_rule} in the group equivariant setting. While the standard Euler characteristic $\calx(C\calf)$ is an alternating sum of numbers of cells by dimension, the components of $\hat{\calx}(C\calf)$ are alternating sums of {\it symmetric force chains} detailed in numerous previous works \cite{fowler2000symmetry, CONNELLY2009762, connelly_guest_2022}. The symmetric Maxwell rule is useful for quickly detecting self-stresses of different symmetry types, which are undetectable using standard non-symmetric counts.

\begin{example}
    To demonstrate the symmetric Maxwell rule we reexamine the mirror-symmetric Figure~\ref{fig:mirror_gs}(a). Aligned with the trivial representation $\mu^{(1)}$, the chain complex $C^{(1)}\calf$ consists of chains that are fixed by every group action on the underlying space. The space $C^{(1)}_0\calf$ is spanned by fully mirror-symmetric vector assignments to vertices; these are spanned by vertical forces assignments to nodes B and C, as well as mirror-symmetric force pairs to A-C and D-F causing $N_0^{(1)}(\rho) = 2 + 2\cdot 2 = 6$. The space $C_1^{(1)}\calf$ is spanned by symmetric edge-assignments, the dimension of which (for the force cosheaf $\calf$) is equivalent to the number of edge orbits. Because three edges lie along the axis of symmetry and six don't, $N_1^{(1)}(\rho) = 3 + 6/2 = 6$.

    The symmetric Maxwell rule states that for the trivial index $(1)$, the equation $$\hat{\calx}^{(1)}(C\calf, \rho) = \hat{\calx}^{(1)}(H\calf, \rho_{H\calf})$$ evaluates to the following equation
    \begin{equation*}
         N_0^{(1)}(\rho) - N_1^{(1)}(\rho) = 6 - 6 = 0 = N_0^{(1)}(\rho_{H\calf}) - N_1^{(1)}(\rho_{H\calf}) = \dim H_0^{(1)}\calf - \dim H_1^{(1)}\calf
    \end{equation*}
    using  $\dim \mu^{(1)} = 1$. This means that there is an equal number of mirror symmetric self-stresses as kinematic degrees of freedom. Indeed, the single degree of self-stress is mirror-symmetric, pictured in Figure~\ref{fig:mirror_gs}(a), and the mirror-symmetric degree of freedom is vertical translation of the framework.
\end{example}

%%%%%%%%%%%%%%%%%%%%%%%%%%%%%%%%%%%%%%%%%%%%%%%%%%
\subsection{Symmetric Graphic Statics}
%%%%%%%%%%%%%%%%%%%%%%%%%%%%%%%%%%%%%%%%%%%%%%%%%%

We recall that the exact sequence of $G$-chain complexes~\eqref{eq:GS_chainses} is isomorphic to a direct sum of irreducible chain complex components of the form~\eqref{eq:chain_irrid}. For a given irreducible representation of $G$ with index $(j)$, the following diagram commutes with exact rows
\begin{equation*}
    \begin{tikzcd}
        & 0 \ar[r] \ar[d] & C_2^{(j)}\overline{\R^2} \ar[r, "\pi^{(j)}"] \ar[d, "\bdd^{(j)}"] & C_2^{(j)} \calp \ar[r] \ar[d, "\bdd^{(j)}"] & 0 \\
        0 \ar[r] & C_1^{(j)} \calf \ar[r, "\phi^{(j)}"] \ar[d, "\bdd^{(j)}"] & C_1^{(j)} \overline{\R^2} \ar[r, "\pi^{(j)}"] \ar[d, "\bdd^{(j)}"] & C_1^{(j)} \calp \ar[r] \ar[d] & 0 \\
        0 \ar[r] & C_0^{(j)}\calf \ar[r, "\phi^{(j)}"] & C_0^{(j)} \overline{\R^2} \ar[r] & 0 & \\
    \end{tikzcd}
\end{equation*}
summarizing the relationship between the symmetric chains (force loadings and positions) relevant in graphic statics.

Because $H_1 X=0$ and thus $H_1 \overline{\R^2} = 0$, it follows that $H_1^{(j)} \overline{\R^2} = 0$ for each $(j)$. The homology spaces $H_2 \overline{\R^2}$ and $H_0 \overline{\R^2}$ consist of constant vector assignments to every face and vertex in $X$. Consequently, the $G$-modules $(H_i \overline{\R^2},\eta_{H_i \overline{\R^2}})$ and $(\R^2, \tau)$ are equivalent for $i=0,2$ and the isomorphism $\R^2 \to H_i \overline {\R^2}$ is the diagonal map.

Utilizing Lemma~\ref{lem:split_les}, the long exact sequence corresponding to an irreducible representation $\mu^{(j)}$ splits into two exact sequences
\begin{equation}\label{eq:split_seq1}
    0 \to H_2^{(j)} \overline{\R^2} \xrightarrow{\pi^{(j)}} H_2^{(j)} \calp \xrightarrow{\vartheta^{(j)}} H_1^{(j)} \calf \to 0
\end{equation}
\begin{equation*}
    0 \to H_1^{(j)} \calp \xrightarrow{\vartheta^{(j)}} H_0^{(j)} \calf \xrightarrow{\phi^{(j)}} H_0^{(j)} \overline{\R^2} \to 0,
\end{equation*}
where the former exact sequence describes the $\mu^{(j)}$-symmetric components of the graphic statics relation, described previously in Example~\ref{ex:graphic_statics}.

\begin{theorem}[Symmetric planar 2D graphic statics] \label{thm:G_GS}
    Let $(X, \alpha, p)$ be a planar $G$-framework in $\R^2$. For each irreducible representation $\mu^{(j)}$ of $G$, there is an isomorphism between $\mu^{(j)}$-symmetric self-stresses and $\mu^{(j)}$-symmetric realizations of the dual $G$-cell complex up to $\mu^{(j)}$-translational symmetry.
\end{theorem}

\begin{proof}
    The first exact sequence~\eqref{eq:split_seq1} reduces to
    \begin{equation}\label{eq:eq_graphic_statics}
        0 \to (\R^2)^{(j)} \xrightarrow{\pi} H_2^{(j)} \calp \xrightarrow{ \vartheta^{(j)}} H_1^{(j)} \calf \to 0
    \end{equation}
    meaning $\vartheta^{(j)}:H_2^{(j)} \calp / \pi^{(j)}(\R^2)^{(j)} \iso H_1^{(j)} \calf$ is an isomorphism. The image $\pi^{(j)}(\R^2)^{(j)}$ consists of the $\tau^{(j)}$-constant vector assignments to all faces of $X$ (constant translational assignments to dual vertices).
\end{proof}
Because $\tau$ is of dimension 2, $\tau^{(j)}$ is non-trivial for at most two indices $(j)$. When $\tau^{(j)}= 0$, the connecting homomorphism $\vartheta^{(j)}$ is then an isomorphism. See Example~\ref{ex:irrep_tau} for some examples of $\tau^{(j)}$.

\begin{figure}[htp]
    \centering
    \includegraphics[scale = 0.9]{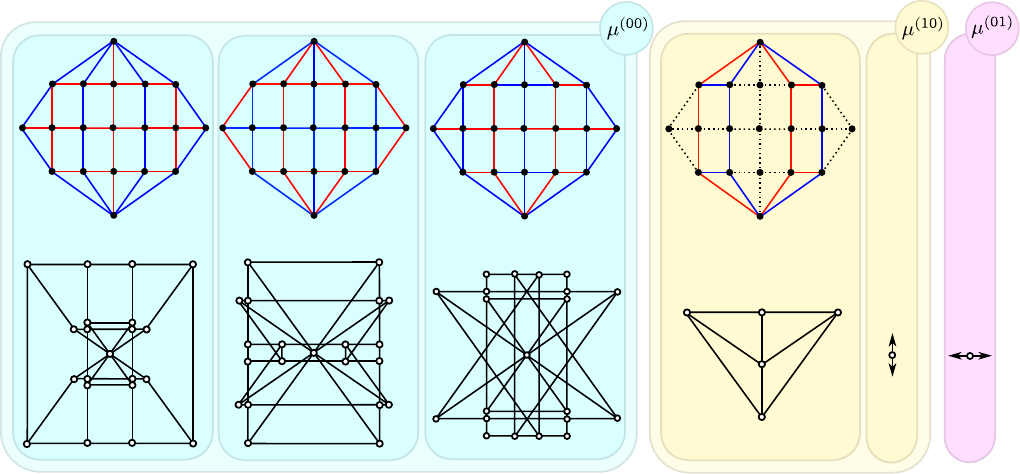}
    \caption{We revisit the self-stresses and reciprocal figures of Figure~\ref{fig:klien_recip} through the symmetric graphic statics Theorem~\ref{thm:G_GS}. Each irreducible representation of index $(j)$ is colored and generator pairs of both $H_1^{(j)}\calf$ and $H_2^{(j)} \calp$ are matched. There are four degrees of self-stress and six degrees of reciprocal figures. In each pair, the self-stress over the primal framework is pictured above with red edges in compression, blue edges in tension, and dashed edges with zero force. The corresponding reciprocal figure is pictured below. The trivial reciprocal translations on the right are global vector assignments in the image $G$-module $\pi(\R^2, \tau)$ of $H_2 \calp$ following the exact sequence~\eqref{eq:eq_graphic_statics}. Because the standard representation $\tau$ decomposes as $\mu^{(01)}\oplus\mu^{(10)}$ following Example~\ref{ex:irrep_tau}, these generate the global translations $\pi(\R^2)^{(01)}$ and $\pi(\R^2)^{(10)}$. The fourth irreducible representation $\mu^{(1,1)}$ is assigned only trivial cycles.}
    \label{fig:klein_gs}
\end{figure}

\begin{example}
    A truss with one degree of self-stress and $D_2$ mirror symmetry is pictured in Figure~\ref{fig:mirror_gs}(a). Its three degrees of reciprocal diagrams are either $\mu^{(1)}$- or $\mu^{(2)}$-symmetric, which we now describe.
    
    The group $D_2$ has two irreducible one-dimensional representations, the trivial representation $\mu^{(1)}$ and the representation $\mu^{(2)}$ that takes value $-1$ on the reflection $s\in D_2$. The self-stress pictured in Figure~\ref{fig:mirror_gs}(a) corresponds to the former, with $H^{(1)}_1 \calf = H_1 \calf$ dimension 1. The representation $\tau^{(1)}$ is one-dimensional, and the graphic statics sequence~\eqref{eq:eq_graphic_statics} with superscript (1) has isomorphic entries to
    \begin{equation*}
        0 \to \R \to H_2^{(1)} \calp \xrightarrow{\vartheta^{(1)}} H_1^{(1)} \calf \to 0
    \end{equation*}
    Here $H_2^{(1)} \overline{\R^2} \iso \R$ consists of the left-right translations of the reciprocal (b). One dimension of $H_2^{(1)}\calp$ corresponds to this horizontal translation, while the other scales the diagram pictured in (b) corresponding to the self-stress in (a). The exact sequence~\eqref{eq:eq_graphic_statics} for superscript (2) is equivalent to
    \begin{equation*}
        0 \to \R \to H_2^{(2)} \calp \to 0
    \end{equation*}
    where $H_2^{(2)} \overline{\R^2} \iso \R$ consists of the up-down translations of the reciprocal, which are inverted by the action $\eta_f(s) = -\tau(s)$ over the reflection $s$ at every face/dual vertex.
\end{example}

\begin{remark}
    Figures~\ref{fig:klein_gs},~\ref{fig:star_recip}, and~\ref{fig:flower_recip} picture symmetric frameworks with their reciprocal diagrams grouped by irreducible representation. Splitting the homology space $H_2 \calp$ into its irreducible $G$-submodules was done by the following process. For each irreducible representation $\mu^{(j)}$, we form the linear map
    \begin{equation}
        {\bf \gamma}^{(j)} := \sum_{g\in G} \chi(\mu^{(j)})(g) \cdot \eta/\rho(g): C_2 \calp \to C_2 \calp.
    \end{equation}
    Then $\gamma^{(j)}$ has image $C_2^{(j)} \calp$, and when restricted to homology cycles the map $\gamma^{(j)}|_{H_2 \calp}$ has the desired image $H_2^{(j)} \calp$.
\end{remark}

\begin{figure}
    \centering
    \includegraphics{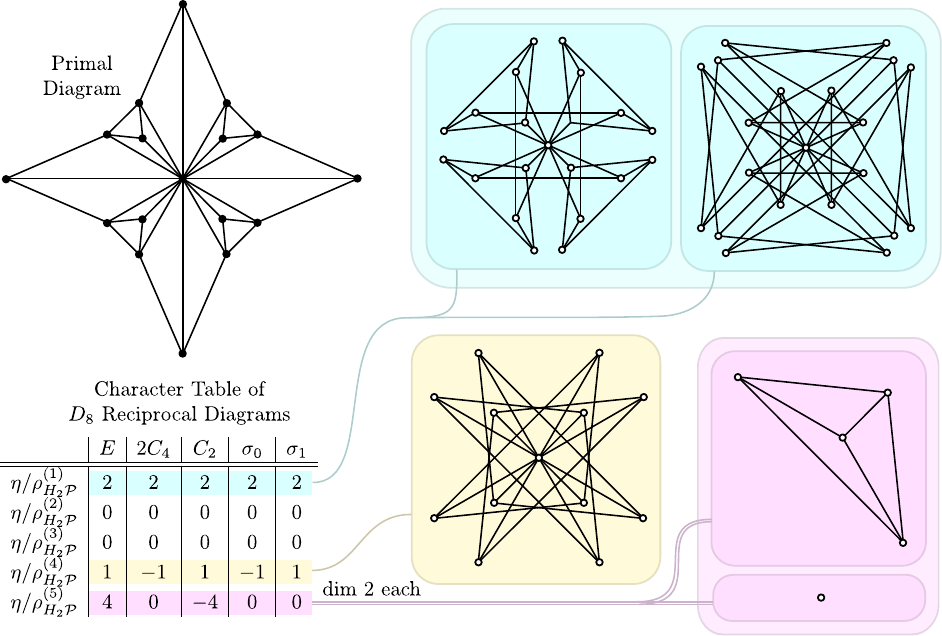}
    \caption{A framework with $D_8$ dihedral symmetry and a family of  reciprocal frameworks corresponding to irreducible representations of $D_8$. The last irreducible representation $\mu^{(5)}$ has dimension two. Then, with the space $H_2^{(5)} \calp$ being four dimensional the representation $\eta/\rho^{(5)}_{H_2 \calp}$ consists of two factors of $\mu^{(5)}$. The result is that the four dimensions of reciprocal diagrams are grouped into pairs of two each.}
    \label{fig:star_recip}
\end{figure}

\begin{figure}
    \centering
    \includegraphics{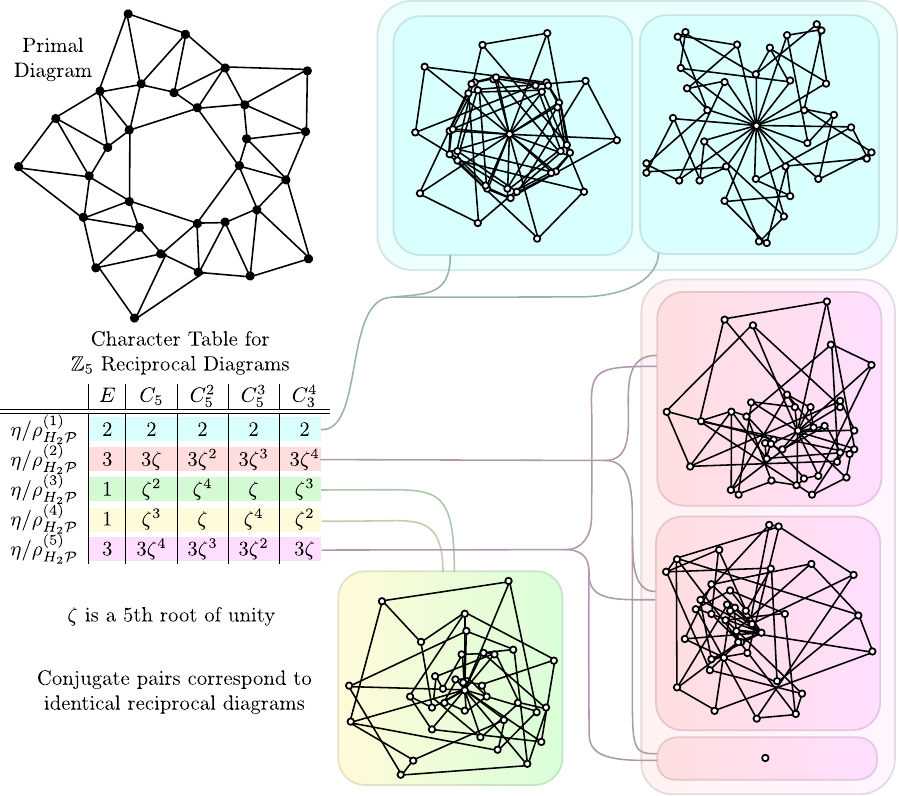}
    \caption{A framework with $5$-fold rotational symmetry and a family of reciprocal frameworks corresponding to irreducible representations of $\Z_5$. Note that conjugate pairs of representations correspond to identical reciprocal diagrams. The coordinates of dual vertices are a pair of complex numbers with zero imaginary component, thus taking the complex conjugate of a coordinate has no effect.}
    \label{fig:flower_recip}
\end{figure}

%%%%%%%%%%%%%%%%%%%%%%%%%%%%%%%%%%%%%%%%%%%%%%%%%%
\section{Conclusion and Future Work}\label{sec:future}
%%%%%%%%%%%%%%%%%%%%%%%%%%%%%%%%%%%%%%%%%%%%%%%%%%

We have developed graphic statics for symmetric frameworks, decomposing the classical self-stress to reciprocal diagram correspondence by the irreducible representation of the finite group. These reciprocal diagrams are useful as both a geometric visualization of the internal stresses and as a design tool for the framework (and its polyhedral liftings). Through the symmetric graphic statics result -- Theorem~\ref{thm:G_GS} -- frameworks can now be designed graphically along the symmetry of their internal stresses.

It is straightforward to generalize the methods here to frameworks with exterior loadings as well as polyhedral lifts of self-stressed frameworks with cyclic and dihedral symmetry. This leads to the natural question of formulating equivariant cosheaves over {\it periodic} stressed frameworks, considered as the lift of a toroidal lifting \cite{ErLin} to the simply connected cover $\R^2$. The group characters and representations are then intimately related to the discrete Fourier transform by the Peter-Weyl theorem, indicating a deep connection with methods in harmonic analysis \cite{mackey1980harmonic}. In higher dimensions, there are potential extensions of this work to vector graphic statics \cite{d2019vector}, 3D-graphic statics \cite{ReciprocalAkbarzadeh2016, AlgebraicHablicsek2019}, and beyond. Higher order symmetry has already been utilized in designing tensegrity structures to great effect \cite{connelly_guest_2022}.

The kinematics and dynamics of symmetric structures can likewise be investigated homologically and equivariantly. The instantaneous velocities of a framework (treated as a pinned linkage) are homology classes \cite{cooperband2023reciprocal} and could be treated with the cosheaf method. This may be particularly useful for modeling the folding of symmetric origami, which have found remarkable applications in deployable aerospace structures, biomedical devices, and metamaterials \cite{lu2023origami}. We note that after applying such a velocity for some finite time, the underlying framework symmetry may change.

It is open to study the actions of subgroups $H \leq G$ and the relations between representations and characters under subgroup restriction or Brauer induction, such as the approach in \cite{boltje1994identities}. We suspect one can define a sheaf theoretic analogue of Bredon homology \cite{MayEquivariant}; however this would require extra structure outside of the scope of this paper. Moreover, it is unclear how to interpret equivariant homology groups of a cosheaf, but we  believe that it is worthy of further investigation.

There has been much previous work on the {\it quotient framework} of a symmetric framework (with possible self-loops) under the group action in question \cite{SWbook, SCHULZE2023112492}. One consequence of this work is that the self-stresses of quotient frameworks are identified as an {\it equivariant homology} space. It is well known that when $G$ is a free action, a $G$-sheaf on $X$ is equivalent to an ordinary sheaf on $X / G$ \cite{EquivShvAndFunctors}. We suspect that methods from equivariant homology theory \cite{jones1987cyclic} can be utilized towards the understanding of quotient frameworks and structures.

\section*{Acknowledgements}

Bernd Schulze was partially supported by the ICMS Knowledge Exchange Catalyst Programme. Zoe Cooperband and Miguel Lopez were supported by the Air Force Office of Scientific Research under award number FA9550-21-1-0334. The authors are grateful for their illuminating discussions with Robert Ghrist, Jakob Hansen, Maxine Calle, and Mona Merling.

\printbibliography

\end{document}